\documentclass [10pt] {article} 
\usepackage[affil-it]{authblk}

\usepackage{tikz}
\usepackage{graphicx}
\usepackage{amssymb}
\usepackage{amsmath}
\usepackage[section]{placeins}
\usepackage{mathrsfs}  
\usepackage{multirow}
\usepackage{amsthm}
\usepackage{float}
\usepackage{bbm}
\usepackage{colortbl}
\usepackage{xcolor}
\usepackage{mwe}
\usepackage{subfig}
\usepackage{mathtools}
\usepackage{hyperref}
\newtheorem{thm}{Theorem}[section]
\newtheorem{prop}[thm]{Proposition}
\newtheorem{lem}[thm]{Lemma}

\newtheorem{rem}[thm]{Remark}

\newcommand{\be}{\begin{equation}}
\newcommand{\ee}{\end{equation}}

\usepackage{datetime} 
\usepackage[T1]{fontenc}
\usepackage[english] {babel}

\begin{document}
%
%
%
%
\baselineskip=15pt

\title{A symmetric low-regularity integrator for the \\ nonlinear Schr\"odinger equation}

\author{Yvonne Alama Bronsard
	\thanks{\texttt{Sorbonne Universit\'e, Laboratoire Jacques-Louis Lions (LJLL), F-75005 Paris, France.\newline
	email: yvonne.alama\_bronsard@sorbonne-universite.fr}}}

\maketitle

\begin{abstract}
We introduce and analyze a symmetric low-regularity scheme for the nonlinear Schr\"odinger (NLS) equation 
beyond classical Fourier-based techniques.
We show fractional convergence of the scheme in $L^2$-norm, from first up to second order, both on the torus $\mathbb{T}^d$ and on a smooth bounded domain $\Omega \subset \mathbb{R}^d$, $d\le 3$, equipped with homogeneous Dirichlet boundary condition. The new scheme allows for a symmetric approximation to the NLS equation in a more general setting than classical splitting, exponential integrators, and low-regularity schemes (i.e. under lower regularity assumptions, on more general domains, and with fractional rates). We motivate and illustrate our findings through numerical experiments, where we witness better structure preserving properties and an improved error-constant in low-regularity regimes.
\end{abstract}

\section{Introduction}
We consider the nonlinear Schr\"odinger (NLS) equation,
\be \label{NLS}
i  \partial_t u(t,x) = - \Delta u(t,x)  + \vert u(t,x)\vert^2 u(t,x),\quad (t,x) \in \mathbb{R}\times  \Omega 
\ee
with
$\Omega \subset \mathbb{R}^d$ or $\Omega = \mathbb{T}^d$, $d \le 3$, and an initial condition
\begin{equation}
\label{init}
u_{|t=0}= u_{0}.
\end{equation}
When $\partial \Omega \neq \emptyset$, we assume that $\Omega$ is a smooth bounded domain and we assign homogeneous boundary conditions which will be encoded in the choice of the domain of the operator $i\Delta$.
In the convergence analysis we will consider either periodic or homogeneous Dirichlet boundary conditions. Nevertheless, one could also consider different types of boundary conditions such as homogeneous Neumann boundary conditions by defining the functional spaces accordingly (see Section \ref{sec:notation}). 

Throughout this article we will be interested in providing a reliable approximation of \eqref{NLS} when the initial data $u_0$ are non-smooth, in the sense that they belong to Sobolev spaces of low order. Namely, we will be interested in studying numerical schemes which approximate the time dynamics of \eqref{NLS} at low regularity.
The numerical study of low-regularity approximations to nonlinear evolution equations has gained lots of attention in the past years, and numerous contributions have been made in this direction. The first results were established on the torus for the Korteweg-De Vries (KdV) equation and then the NLS equation with the pioneering works of \cite{HofS} and \cite{OS1}. These results could be further greatly extended, see for instance \cite{WZhao22,mKdV,RSKDV,LW22} and \cite{ORS2,WuYao22,LW21,CLLsecondNLS} for the KdV and NLS equations respectively. 
 More types of dispersive equations could be dealt with, including for example the Dirac equation \cite{S_Dirac} or the Klein-Gordon equation (\cite{CS_KG}), and a general framework for constructing low-regularity approximations up to arbitrary order and for a class of dispersive equations on the torus was obtained in \cite{BS}.

The construction of these time integrators (called {\it resonance-based schemes}, {\it exponential-type low-regularity integrators}, or {\it Fourier integrators}) strongly depended on Fourier-based expansions, and hence were restricted to periodic boundary conditions.
Recently, this restriction was withdrawn to treat more general domains $\Omega \subset \mathbb{R}^d$ and boundary conditions,  as well as more general nonlinearities (see \cite{RS,YGP,ABBSBdd,lowNeumann, NS}).

The general aim of low-regularity integrators is that they converge under lower regularity assumptions, contrarily to classical methods (see \cite{ABBSBdd} for a comparative analysis on general smooth domains).
Their major drawback is that they do not preserve the geometric structure of the underlying system. The NLS equation \eqref{NLS} is time reversible, meaning that $\overline{u(-t,x)}$ is again solution of \eqref{NLS}, and both the density and the energy are conserved quantities:
\begin{align}\label{conserved}
\left\{\begin{array}{ccc}
\|u(t)\|_{L^2} &=& \|u_0\|_{L^2},\\ 
E(t) &=& E_0,
\end{array}
\right.\quad  t\in I,
\end{align}
where $I$ is the interval of existence of the solution, and where for $H^1$-solutions we have
$$
E(t) = \frac{1}{2} \int |\nabla u|^2(t,x) dx + \frac{1}{4}\int |u|^4(t,x) dx.
$$
Hence, when designing a numerical scheme it is natural to take into account both of these conserved quantities, and to retain (as much as possible) these properties also on the discrete level by introducing so-called structure preserving schemes, see \cite{HLW} for an extensive introduction on the subject. The latter has received great interest thanks to their good long-time near-preservation of the actions of the integrable properties of the equation, and have been successfully studied in the past for the approximation of the NLS equation \eqref{NLS}. 
Examples of such schemes are splitting schemes (\cite{L}, \cite{Faou}), relaxation type schemes (\cite{Besse}), symmetric exponential integrators (\cite{SymmExp}) or Crank-Nicolson Galerkin methods (\cite{Peterseim}), just to name a couple of them. 
For an overview of symmetric methods for NLS see \cite{SymmExp, Faou}.
While these classical structure preserving schemes provide excellent approximations to smooth solutions in general even up to long times, they often break down and lead to severe loss of convergence for non-smooth solutions. Low-regularity integrators which are suited for non-smooth solutions on the other hand do not preserve the structure of the underlying equation. The natural question which thus arises is: {\it What about low-regularity structure preserving schemes for solving the NLS equation \eqref{NLS}?} 
Only very little is known in this direction, see the work of \cite{GK} on the KdV equation, \cite{WZ22} on the cubic Klein-Gordon equation, and \cite{VMS} on the isotropic Landau--Lifschitz equation.
Also worth to be mentioned is the work of \cite{WuYao22} which introduces for the first time a first-order Fourier integrator for the NLS equation \eqref{NLS} set on $\mathbb{T}$ which almost conserves the mass.

In this article, we introduce a {\it symmetric} low-regularity integrator for solving the NLS equation \eqref{NLS} which allows for low-regularity approximation while maintaining good long-time preservation of the two conserved properties \eqref{conserved} on the discrete level. We carry-out a rigorous convergence analysis in $L^2(\Omega)$ on smooth domains $\Omega \subset \mathbb{R}^d$ and obtain improved error estimates at low-regularity compared to classical symmetric methods. Our numerical findings not only show better structure preservation properties but also show a much better error constant at low-regularity than previously proposed methods (see Figure \ref{fig:cv}). 

In the finite dimensional ODE setting it is well-known that symmetric methods are of even order. In the context of PDEs this is a much more delicate question as convergence is met only when sufficient regularity assumptions are imposed on the solution. Thanks to the gain of symmetry, we show second order $L^2$-convergence of the symmetric scheme under less regularity assumptions than what is required by classical symmetric schemes (\cite{L, Besse, Peterseim}), while asking for slightly more regularity than asymmetric second-order low-regularity schemes (\cite{YGP}) which however do not preserve the structure of the system (see Figures \ref{fig:mass}  and \ref{fig:energy}). Optimal first order low-regularity convergence rates could be obtained.
 See Section \ref{sec:results} for a detailed discussion on the subject.

The scheme we present here is based on the first-order low-regularity scheme first introduced in \cite{OS1} which is given by,
\begin{equation}\label{LRexpl}
\Phi_\tau(u^n) := e^{i\tau \Delta}\left( u^n - i \tau (u^n)^2 \varphi_1(-2i \tau\Delta)\overline{u^n} \right), \quad u^0 = u_0,
\end{equation}
where $\varphi_1(z)=\frac{e^{z}-1}{z}$, and $\tau$ is the time step.
 In order to symmetrize the above scheme we introduce the adjoint method as the map
$$
{\hat \Phi_\tau} = \Phi^{-1}_{-\tau},
$$
and compute (see \cite{HLW})
$$
{\hat \Phi_{\tau/2}} \circ \Phi_{\tau/2}.
$$
This yields the following implicit symmetric low-regularity scheme,
\begin{align}\label{num}
u^{n+1} = \varphi^\tau (u^n)&=  e^{i\tau\Delta} u^n - i \frac{\tau}{2} e^{i\tau\Delta} \left( (u^n)^2 \varphi_1(-i\tau\Delta) \overline{u^n} \right) - i\frac{\tau}{2} \left( (u^{n+1})^2\varphi_1(i\tau \Delta)\overline{u^{n+1}} \right)\\ \nonumber
&= e^{i\tau\Delta} u^n + \psi_{E}^{\tau/2}(u^n) + \psi_{I}^{\tau/2}(u^{n+1})\\ \nonumber
&= e^{i\tau\Delta} u^n + \Psi^\tau(u^n,u^{n+1}),
\end{align}
which satisfies the discrete analogue of the time-reversible property of \eqref{NLS}.

We highlight the properties which the scheme \eqref{num} inherits through numerical experiments, where we couple the time-integrators with the standard Fourier pseudo-spectral method which encodes periodic boundary conditions.
The case of homogeneous Dirichlet boundary conditions remains very similar, yet for completeness we also include a convergence plot in this case where we expand the solution as a sine series expansion. First, in the case of periodic boundary conditions, we observe in Figures \ref{cv:a} and \ref{cv:b} the favorable convergence properties of the scheme \eqref{num} for $H^1$ and $H^2$ data respectively. 
We notice that the error constant of the symmetric scheme is much better than the asymmetric first-order  low-regularity integrator (Low-reg 1), and is also better than the asymmetric second-order low-regularity integrator (Low-reg 2). Figure \ref{cv:c} similarly shows the favorable convergence behaviour when considering homogeneous Dirichlet boundary conditions.
Secondly, we study in Figures \ref{fig:mass} and \ref{fig:energy} the structure preserving properties of the new symmetric low-regularity integrator \eqref{num} against previous asymmetric low-regularity integrators. We witness that the asymmetric first and second-order low-regularity integrators (Low-reg 1, Low-reg 2) are unable to preserve the density and energy (see \eqref{conserved}), whereas the symmetric integrator \eqref{num} appears to nearly-preserve both conserved properties over long-times. 
We note that in the finite dimensional ODE setting a general theory for symmetric methods applied to integrable reversible systems has been established in \cite{HLW} allowing for long-time near-conservation of first-integrals. In the infinite dimensional case the understanding of the long-time behaviour of numerical solutions is an ongoing challenge in the field of geometric integration and few results are known, see for example \cite{Faou, FGP, GL, CHL}.
We expect that it would be possible to prove long-time near-preservation of the density and energy of the scheme \eqref{num} by using the results of \cite{HLW}, and by benefitting of an analysis using modulated Fourier expansions (see \cite{CHL, GL}) or using normal form techniques (\cite{FGP,BG}) to show near-conservation of the energy. This delicate analysis is out of scope for this paper, where here we focus on the low-regularity error estimates on the solution itself.
Finally, we refer to Figure \ref{CPU} for a broad indication of the relative computational cost of each of the three low-regularity integrators, and discuss the added cost of implementing the symmetric implicit scheme \eqref{num}. We observe that the asymmetric second-order low-regularity integrator (Low-reg 2) costs in CPU-time approximately the same as the symmetric integrator \eqref{num}. Whereas when comparing with the asymmetric first-order scheme (Low-reg 1) we have that the improved convergence properties of the scheme \eqref{num} make up for the extra cost of solving the implicit system \eqref{num} at every time step.

\begin{figure}[H]

\begin{minipage}{.5\linewidth}
\centering
\subfloat[]{\label{cv:a}\includegraphics[scale=.40]{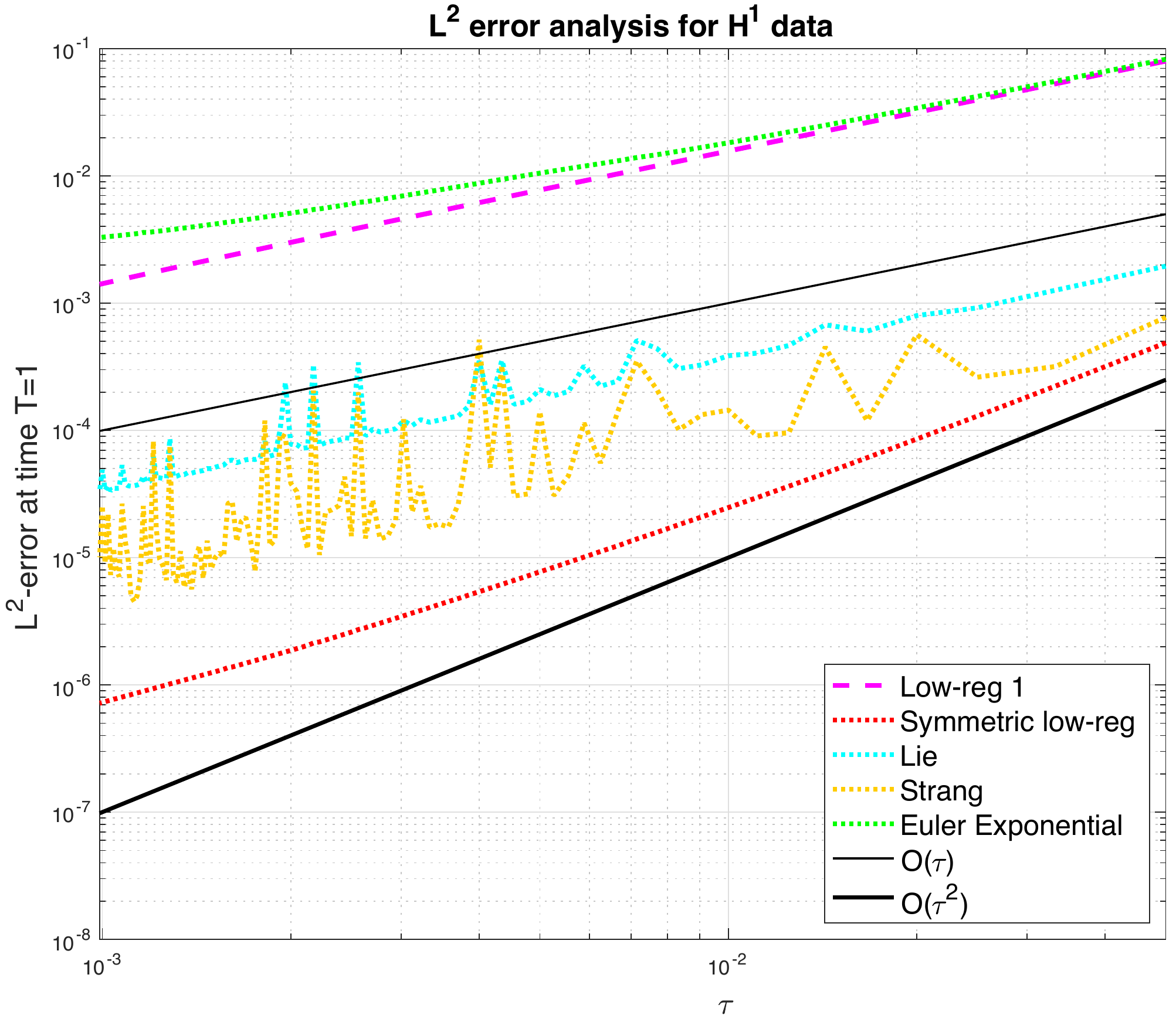}}
\end{minipage}%
\begin{minipage}{.5\linewidth}
\centering
\subfloat[]{\label{cv:b}\includegraphics[scale=.40]{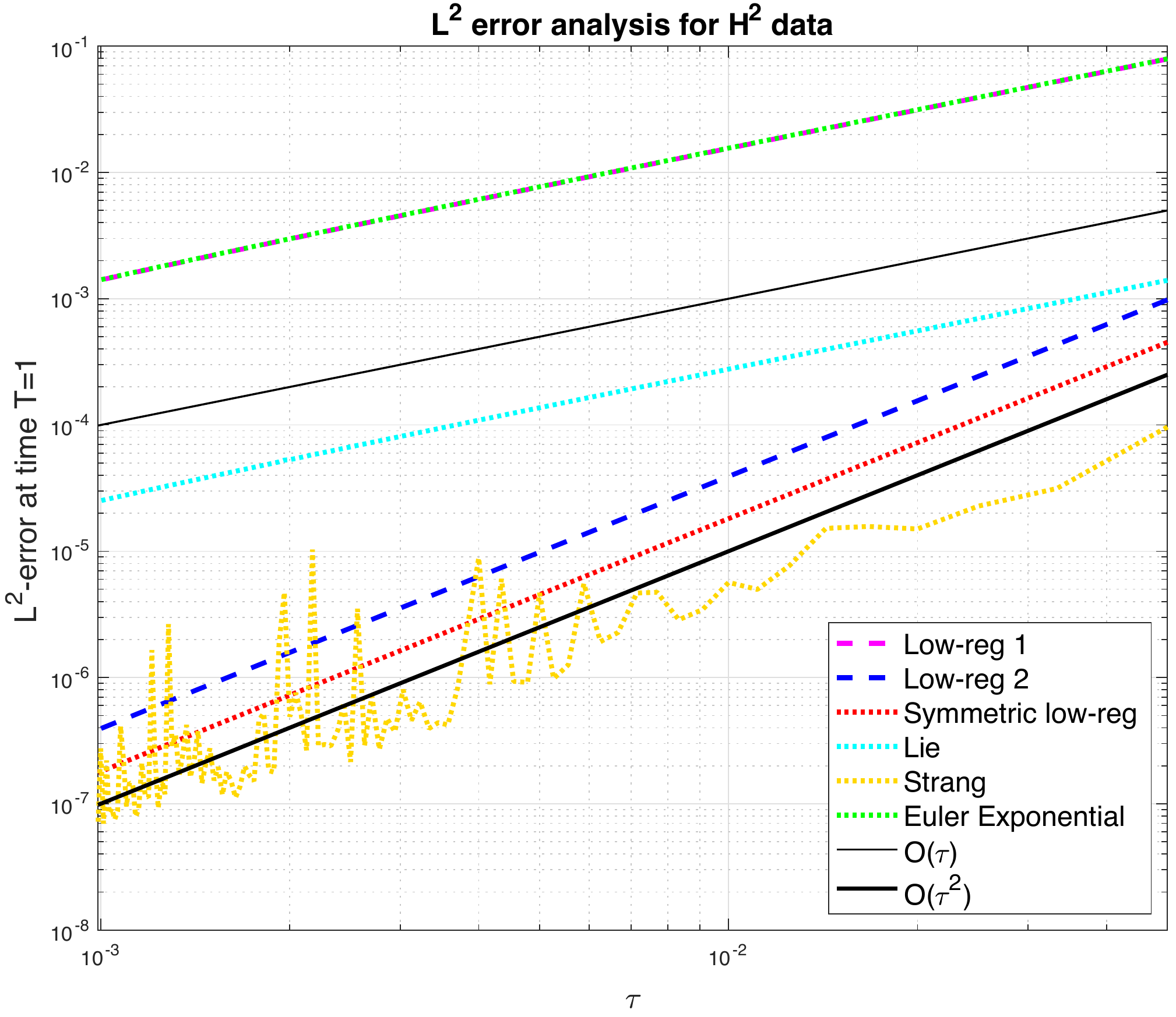}}
\end{minipage}\par\medskip
\centering
\subfloat[]{\label{cv:c}\includegraphics[scale=.43]{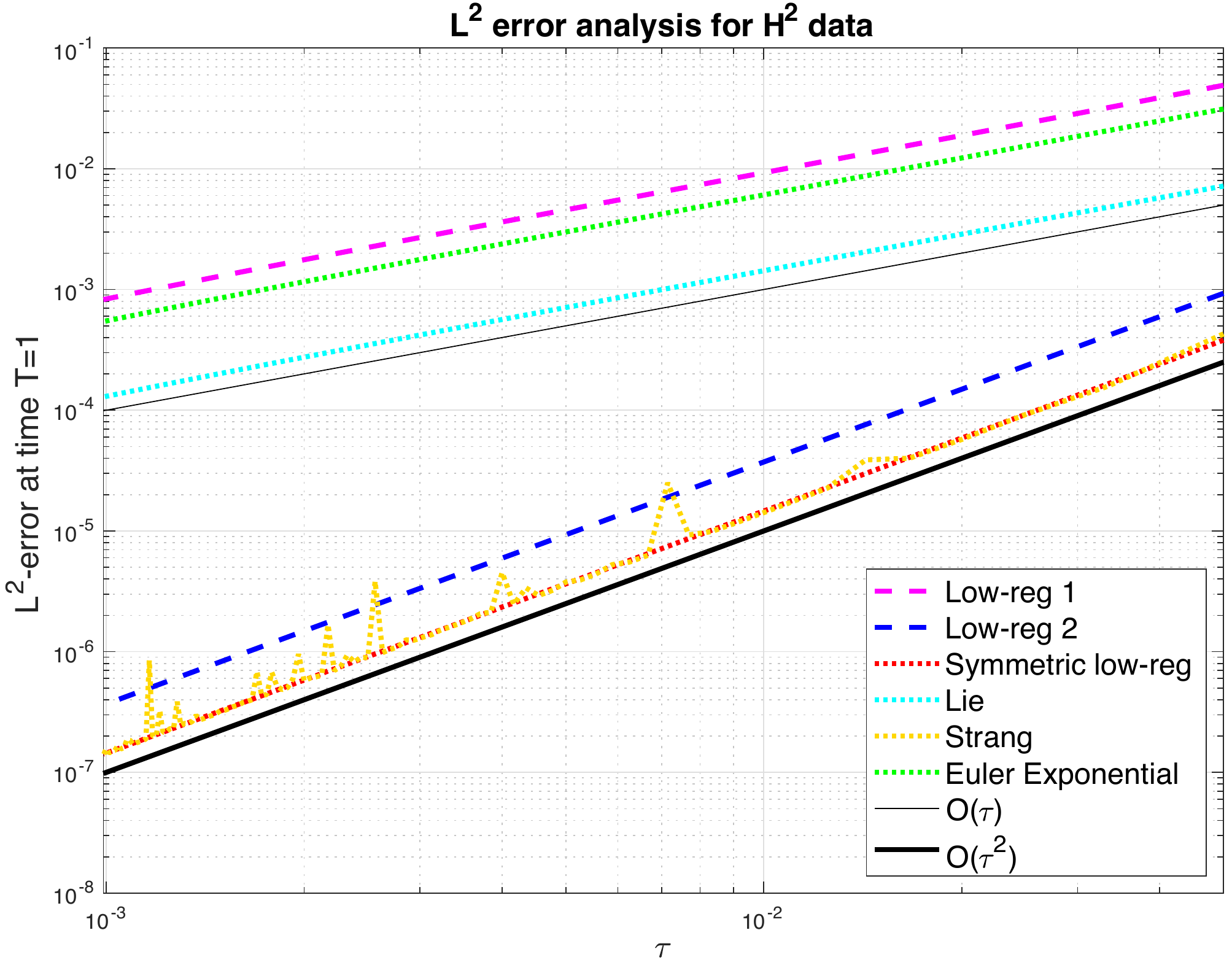}}

\caption{ 
Convergence plot for data in $H^1$ (Figure (a)) and data in $H^2$ (Figure (b) and (c)) of the asymmetric first and second-order low-regularity integrators (pink, dark blue), the symmetric method \eqref{num} (red), the classical Lie splitting, Strang splitting, and Euler Exponential method (light blue, yellow, and green). We observe order reduction of the classical Euler Exponential and Lie and Strang splitting methods (Figure (a), $H^1(\mathbb{T})$-data), and of the Strang splitting method (Figure (b), $H^2(\mathbb{T})$-data, and Figure (c), $H^2([0,1])$-data). Figures (a) and (b) are with periodic boundary conditions, while Figure (c) is with homogeneous Dirichlet boundary conditions. The slopes of the continuous black lines are one and two, respectively. We took the final time $T=1$, and the number of Fourier modes $K=2^{11}$.
}
\label{fig:cv}
\end{figure}
\begin{figure}[H]
\begin{minipage}{.5\linewidth}
\centering
\subfloat[]{\label{mass:a}\includegraphics[scale=.4]{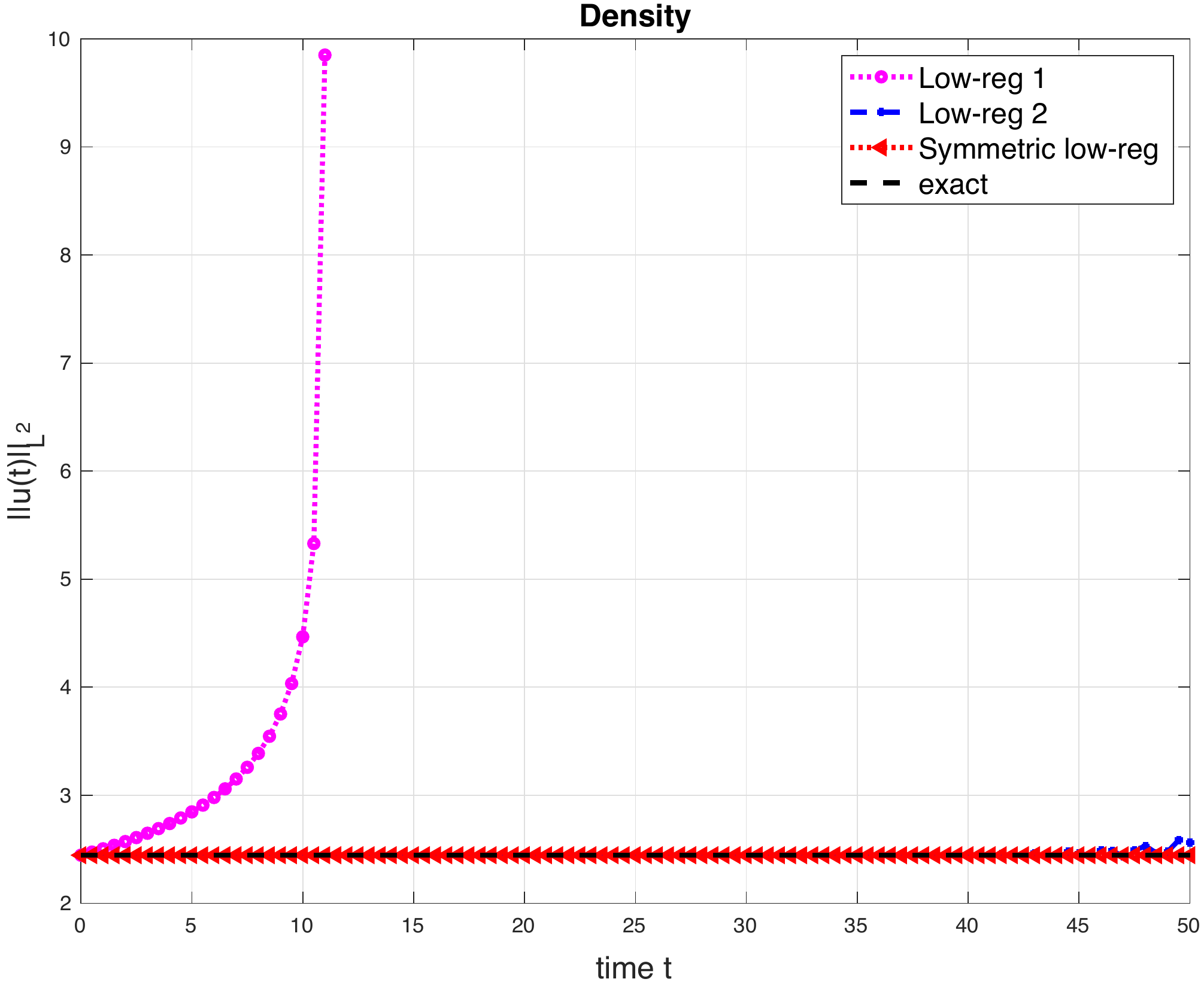}}
\end{minipage}%
\begin{minipage}{.5\linewidth}
\centering
\subfloat[]{\label{mass:b}\includegraphics[scale=.4]{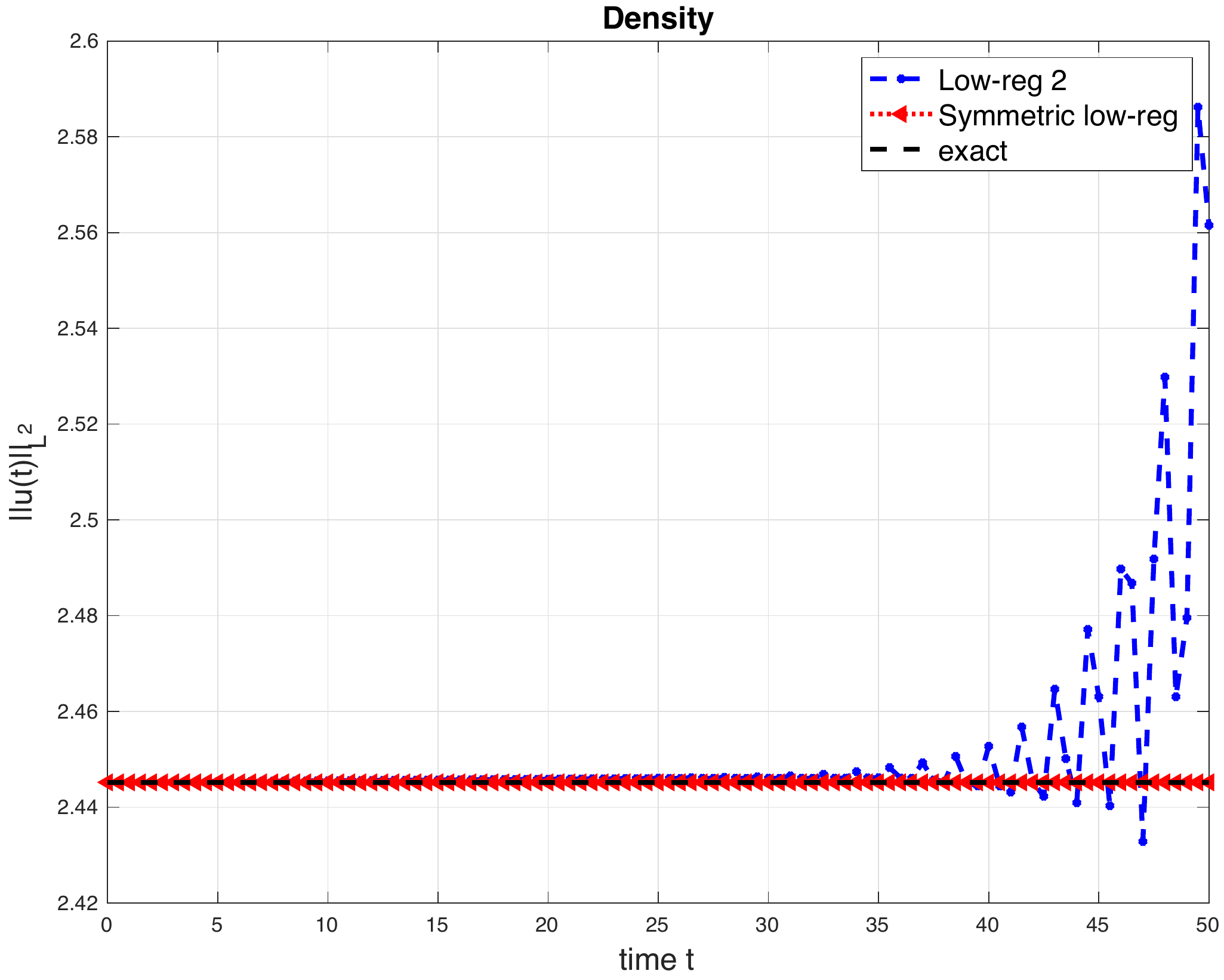}}
\end{minipage}
\caption{
Plot (a) : We graph the $L^2$-norm $\| u(t)\|_{L^2(\mathbb{T})}$ up until time $T=50$ of three low-regularity integrators. The asymmetric first and second-order low-regularity integrators (pink, blue), and the new symmetric low-regularity integrator (red). We also graph the exact value $\| u_0\|_{L^2(\mathbb{T})}$ (black).
Plot (b): We only graph the asymmetric second-order low-regularity integrators (blue), and the symmetric low-regularity integrator (red) together with the exact $L^2$-norm of the initial value (black). We fixed the number of Fourier modes $K=2^{9}$, the time step $\tau = 0.05$, and took an initial data $u_0 \in H^2$.
}
\label{fig:mass}
\end{figure}
\begin{figure}[H]
\begin{minipage}{.5\linewidth}
\centering
\subfloat[]{\label{energy:a}\includegraphics[scale=.4]{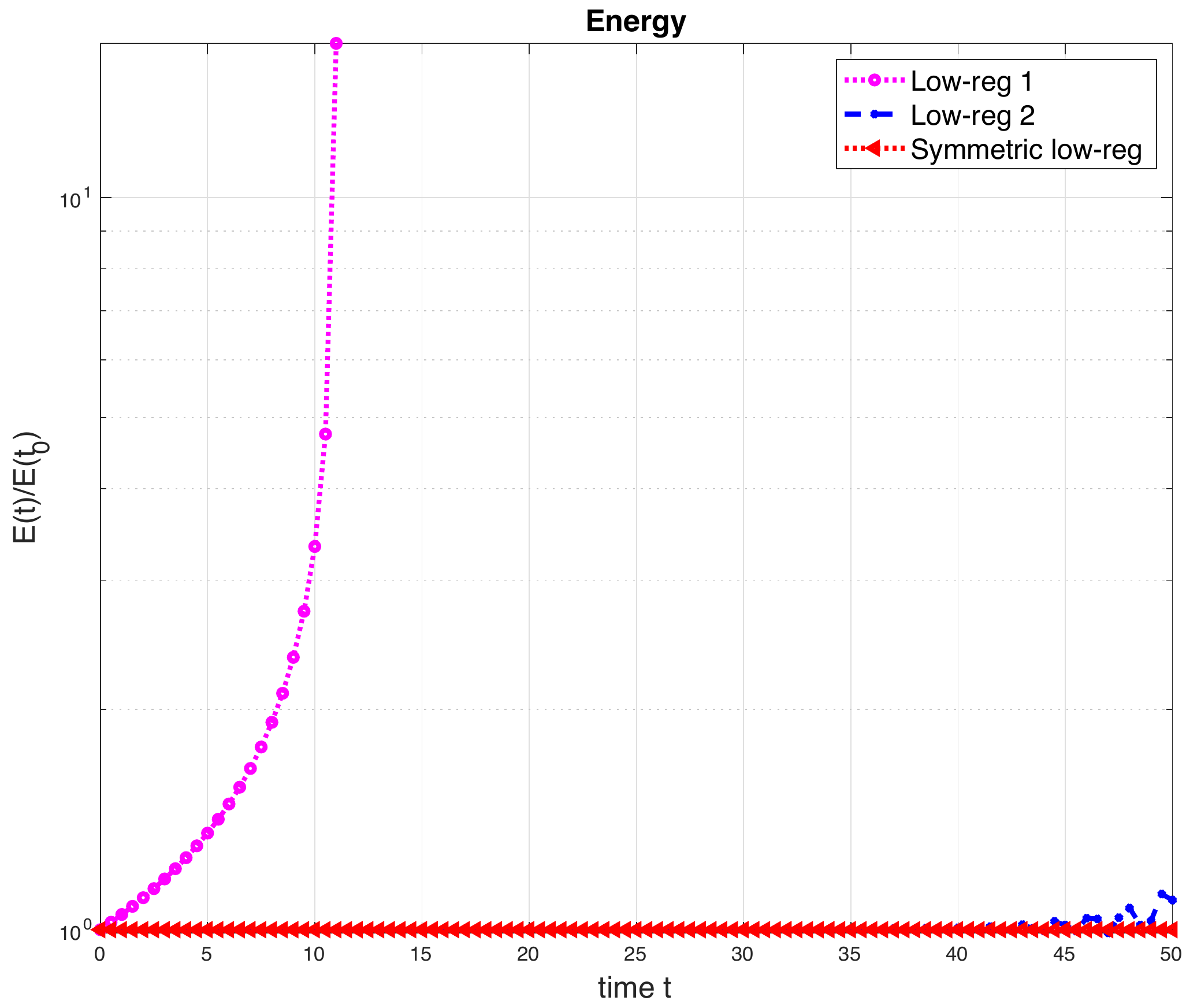}}
\end{minipage}
\begin{minipage}{.5\linewidth}
\centering
\subfloat[]{\label{energy:b}\includegraphics[scale=.4]{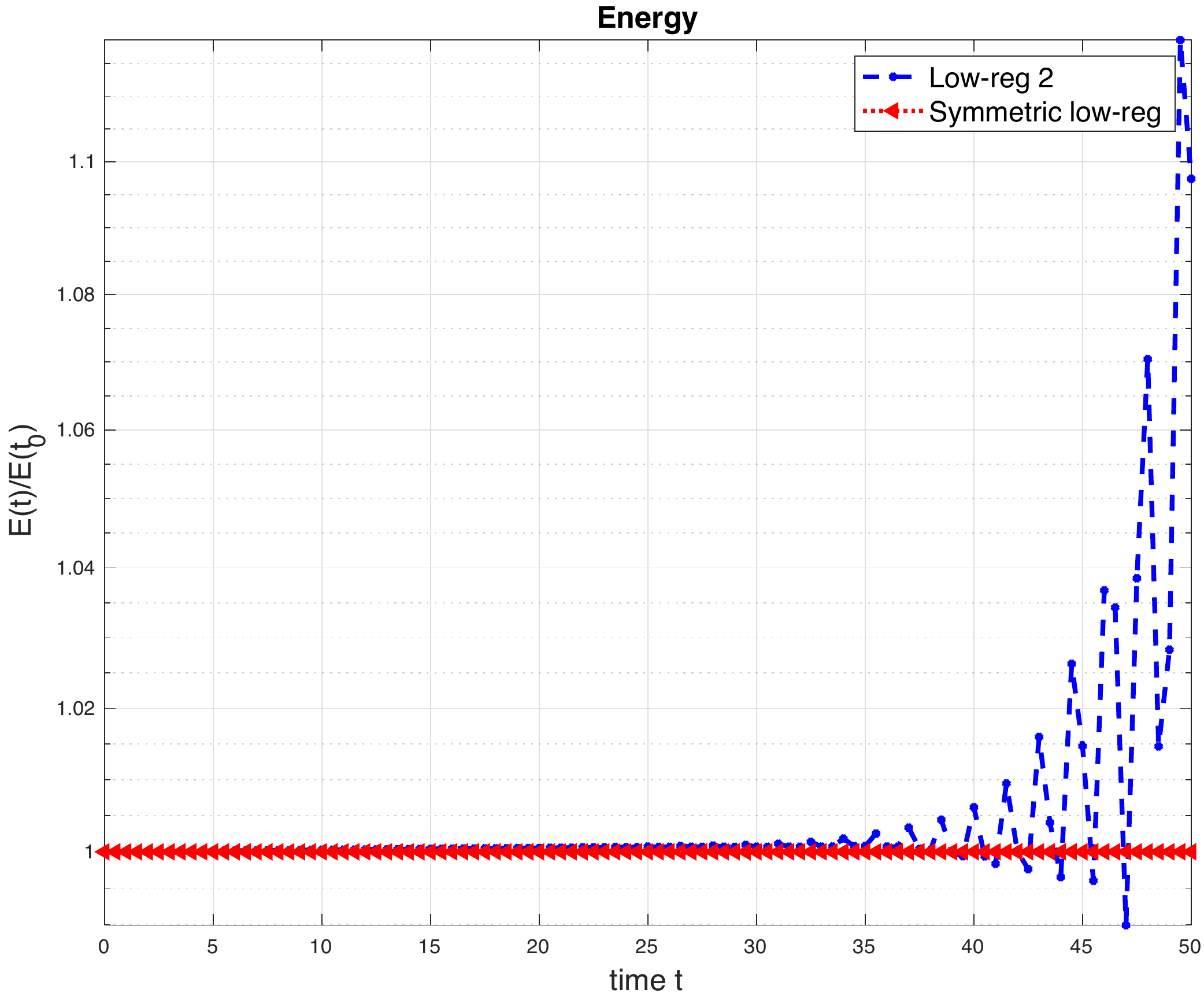}}
\end{minipage}
\caption{
Plot (a) : We graph the relative energy $E(t)/E(t_0)$ (see \eqref{conserved}) up to $T=50$ of the same three low-regularity integrators as in Figures \ref{mass:a} and \ref{mass:b}. 
Plot (b): We only graph the relative energy of the asymmetric second order low-regularity integrators (blue), and the symmetric low-regularity integrator (red). We again fixed the number of Fourier modes $K=2^{9}$, the time step $\tau = 0.05$, and took an initial data $u_0 \in H^2$.
}
\label{fig:energy}
\end{figure}
\begin{figure}[H]
\centering
\includegraphics[scale=.4]{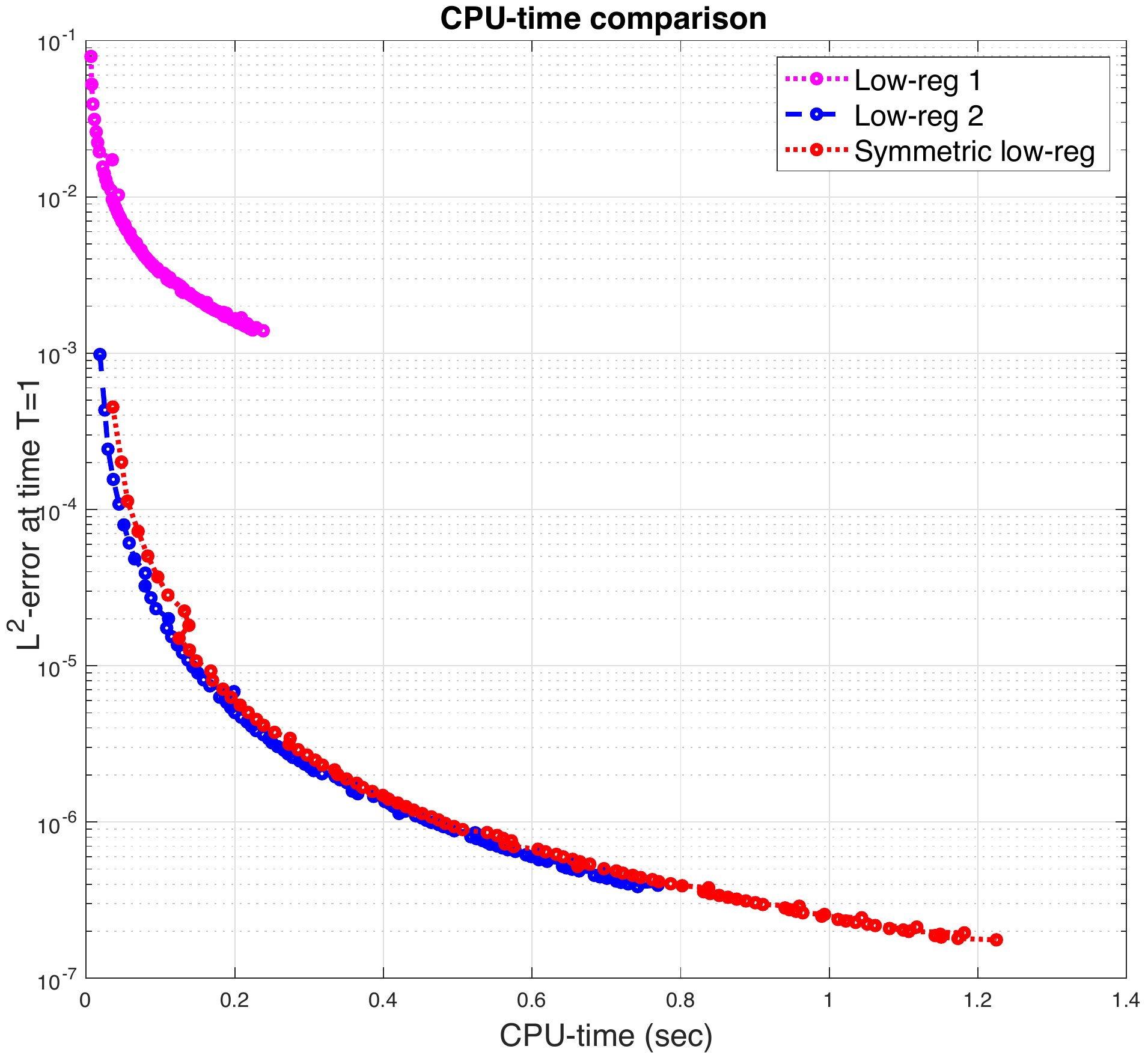}
\caption{
We plot the CPU time versus the $L^2$-error. Namely, we compare the computational cost for running the first and second-order low-regularity scheme (pink, dark blue), with the low-regularity symmetric integrator \eqref{num} (red).  
We took the same parameter values as in Figure \ref{fig:cv}. 
}
\label{CPU}
\end{figure}

This numerical study motivates the use of the scheme \eqref{num}, which conserves better the underlying geometric structure of the equation, exhibits a better error constant, and can be implemented at relatively low additional cost despite its implicit nature.
\begin{rem}[Implicit versus explicit low-regularity schemes]
We make the important remark that unlike the previous (asymmetric) low-regularity integrators \cite{OS1, BS, ORS1, ORS2, RS, YGP,ABBSBdd} the above symmetrized scheme \eqref{num} is an {\it implicit} one. We have witnessed that the implicit nature of the scheme does not adversely affect the computational cost of the method (see Figure \ref{CPU}). Nevertheless, one could query on the necessity of the implicit nature of the symmetric low-regularity scheme. In the case of second-order wave-type equations, instead of considering implicit symmetrized schemes one could study explicit three-time step symmetric schemes using Gautschi-type methods. Indeed,
for the cubic Klein-Gordon equation set on $\mathbb{T}$, an {\it explicit} symmetric {\it three time-step} low-regularity integrator could be obtained by \cite{WZ22}. While this approach is suited to second order equations, by combining the work of \cite{WZ22} together with the uniformly accurate low-regularity integrator \cite{CS_KG} adapted to the non-relativistic regime, an interesting open problem would be to obtain in the non-relativistic limit an {\it explicit} symmetric {\it three time-step} low-regularity approximation to the NLS equation \eqref{NLS}.
\end{rem}

Having motivated the scheme \eqref{num} we now provide the underlying idea behind its construction. To provide a better intuition to the reader we will work in the twisted variable and place ourselves on the torus $\Omega = \mathbb{T}$ in order to make use of Fourier-based expansions (see also \cite{OS1}).

\subsection{Derivation of the scheme}
We switch to the twisted variable $v = e^{-it\Delta}u$. We observe that $v$ satisfies
$$
i\partial_t v = e^{-it\Delta}(|e^{it\Delta}v|^2e^{it\Delta}v), \quad v_0 = u_0.
$$
Equivalently, by integrating the above and mapping Duhamel's formula in Fourier space we have
\be\label{twisted}
v(t_{n+1}) = v(t_n) -i \sum_{k=-k_1+k_2+k_3}e^{ikx}e^{it_n(k^2+k_1^2 - k_2^2-k_3^2)} I^\tau,
\ee
where the oscillatory integral is given by
\be\label{oscil}
I^{\tau} = \int_0^\tau e^{i\omega_1s}h(s)ds,
\ee
and $h(s) = e^{i\omega_2s}g(s)$, $\omega_1 = 2k_1^2$, $\omega_2 = k^2 - k_1^2 -k_2^2 -k_3^2 = -2k_1(k_2+k_3)+k_2k_3$ and 
$$g(s) = {\overline v_{k_1}(t_n+s)}v_{k_2}(t_n+s)v_{k_3}(t_n+s).$$

The central question revolves around making a suitable choice of discretization of the oscillatory integral $I^\tau$, with the aim of minimizing the regularity assumption required by this approximation.
The underlying idea behind the construction of the previous (asymmetric) low-regularity integrators (or {\it resonance-based schemes}) is to choose an approximation of the integral $I^\tau$ which allows for a practical implementation (by not performing exact integration), while optimizing the local error in the sense of regularity. Namely, by recalling that $2k_1^2$ corresponds to second order derivatives in Fourier space while the terms $k_m k_j $ (for $m \not= j$) correspond to product of first order derivatives, the idea is to separate the dominant ($\omega_1$) and lower-order ($\omega_2$) frequencies. The lower-order and non-oscillatory part $h(s)$ is then approximated by a Taylor series expansion centered at $s=0$,
$$
h(s) = h(0) + O(sw_2 g),
$$
and the dominant part $e^{i\omega_1s}$ is integrated exactly. This yields the first-order low-regularity scheme \eqref{LRexpl} with a local error of $O(\tau^2 \partial_x v)$. At low regularity this is more advantageous than classical techniques (such as exponential integrators \cite{ExpInt} or splitting methods \cite{L})
which do not embed the dominant frequency interactions into the scheme and obtain a local error of $O(\tau^2\partial_x^2 v)$. 

The key idea behind obtaining the symmetric scheme \eqref{NLS} is to make a different Taylored discretization of the lower-order and non-oscillatory part $h(s)$. Namely, we again integrate exactly and embed the dominant part $e^{i\omega_1s}$ into the numerical scheme, while this time approximating the non-dominant part in the following symmetric fashion,
\be\label{h-approx}
h(s) \approx h(0)\mathbbm{1}_{[0,\tau/2]} + h(\tau)\mathbbm{1}_{(\tau/2,\tau]}, \quad s \in [0,\tau],
\ee
where $\mathbbm{1}_A$ is the indicator function on the set $A$.
By plugging this approximation for $h$ into the oscillatory integral \eqref{oscil} yields two terms: an explicit and an implicit one. The explicit term is given by,
$$
\int_0^{\tau/2} e^{i\omega_1s}h(0)ds = \tau\frac{e^{i\omega_1\tau/2}-1}{i\omega_1\tau}\overline{v^n_{k_1}}(t_n)v^n_{k_2}(t_n)v^n_{k_3}(t_n) = \frac{\tau}{2} \varphi_1(ik_1^2\tau)\overline{v^n_{k_1}}(t_n)v^n_{k_2}(t_n)v^n_{k_3}(t_n).
$$
Using the definition of the twisted variable, equation \eqref{twisted}, and by mapping the above back to physical space yields the explicit nonlinear term $\psi_{E}^{\tau/2}(u^n)$ in the scheme \eqref{num}. Similarly, one obtains the nonlinear implicit term in \eqref{num} by using the definitions $\omega_1 = 2k_1^2$, $\omega_2 = k^2 - k_1^2 -k_2^2 -k_3^2$ and noticing that
$$
h(\tau) = e^{ik^2\tau }\left( \left(e^{-ik_2^2\tau }v^n_{k_2}(t_{n+1})\right)\left(e^{-ik_3^2\tau }v^n_{k_3}(t_{n+1})\right) \left(e^{-ik_1^2\tau }\overline{v^n_{k_1}}(t_{n+1})\right)\right)
$$
and
\begin{align*}
\int_{\tau/2}^\tau e^{i\omega_1s}h(\tau)ds =& e^{ik^2\tau }\left( \left(e^{-ik_2^2\tau }v^n_{k_2}(t_{n+1})\right)\left(e^{-ik_3^2\tau }v^n_{k_3}(t_{n+1})\right) \int_0^{\tau/2}e^{2ik_1^2(\tau-s)}ds \,
\left(e^{-ik_1^2\tau }\overline{v^n_{k_1}}(t_{n+1})\right)\right)\\
=& e^{ik^2\tau }\left( \left(e^{-ik_2^2\tau }v^n_{k_2}(t_{n+1})\right)\left(e^{-ik_3^2\tau }v^n_{k_3}(t_{n+1})\right)
\frac{\tau}{2} \varphi_1(-i\tau k_1^2) \left(e^{+ik_1^2\tau }\overline{v^n_{k_1}}(t_{n+1})\right)\right).
\end{align*}

We note that a general approach to obtain the approximation \eqref{h-approx} is to first give a symmetric approximation to the  non-oscillatory part $g(s)$ by iterating Duhamel's formula inside $v_k(t_n +s)$ in a symmetric fashion. Namely, $v_k(t_n +s)$ is approximated on $[0,\frac{\tau}{2}]$ by the linear term in the {\it Duhamel formula centered about $s=0$} (yielding the approximation $g(0)$, see \eqref{twisted}). While on $(\frac{\tau}{2},\tau]$, $v_k(t_n +s)$ is approximated
by the linear term in the {\it Duhamel formula centered about $s=\tau$}  (yielding the approximation $g(\tau)$). 
We then proceed by approximating the lower-order oscillatory part $e^{i\omega_2 s}$ in a symmetric fashion. In order to obtain higher-order symmetric low-regularity approximations, we would iterate inside $g(s)$ both of these Duhamel expansions (centered about $s=0$ on $[0,\frac{\tau}{2}]$ and about $s=\tau$ on $(\frac{\tau}{2},\tau]$) up to higher order. The construction of higher-order low-regularity symmetric integrators will be dealt with in future work.

\begin{rem}[Third order local error bound]\label{rem:H3sol}
We make an important point related to the third-order local error structure of the scheme \eqref{num}. Thanks to the symmetry of the scheme \eqref{num} we can expect to have second-order convergence under suitable regularity assumptions on the solution. From the above calculations in Fourier one easily observes that we naturally need three additional derivatives in order to obtain a third-order local error bound of the scheme \eqref{num}.
Indeed, by Taylor expanding around the midpoint one observes that the error induced by the discretization \eqref{h-approx} of $I^{\tau}$ requires the boundedness of a term of the form,
\begin{align}\label{reg}
&\int_{-\tau}^0 e^{i\omega_1(\frac{s+\tau}{2})}\left( e^{i\omega_2(\frac{s+\tau}{2})}-1\right) ds g(0) \\\nonumber
&\quad + \int_0^\tau e^{i\omega_1(\frac{s+\tau}{2})}\left( e^{i\omega_2(\frac{s+\tau}{2})}-e^{i\omega_2\tau}\right) ds g(0).
\end{align}
For the above to yield a third order term one needs to bound a term of order $O(\tau^3 \omega_1\omega_2 \bar{v}_{k_1}v_{k_2}v_{k_3})$ which corresponds in physical space to asking for three additional derivatives on $\bar{v}$. 
This is to be compared with classical (symmetric) schemes which usually have a local error of $O(\tau^3 \partial_x^4 v) $ (see for example \cite{L} for splitting schemes), and to asymmetric resonance-based schemes which merely asks for $O(\tau^3 \partial_x^2 v)$, (\cite{BS, OWY}).
\end{rem}

While we motivated this symmetric low-regularity integrator on a periodic domain, we show that it also allows for a low-regularity approximation on general smooth domains by establishing its convergence at low-regularity (see Section \ref{sec:results}).
Throughout the remainder of this article we will work on general smooth domains $\Omega\subset \mathbb{R}^d$ and make use of semi-group theory to derive our scheme and establish our convergence result on general domains (see Section \ref{sec:errAnal}).
This differs from the first structure preserving low-regularity integrators \cite{GK,VMS} which are restricted to periodic boundary conditions. 

We now enter the main bulk of this paper, which answers the question of what can be rigorously proven on the $L^2$-convergence of the scheme \eqref{num} when set on a general smooth domain.

We state and prove $L^2$-fractional convergence results, from first to second order, both on the torus $\mathbb{T}^d$ and on a smooth bounded domain $\Omega \subset \mathbb{R}^d$, under moderate regularity assumptions on the solution $u$. These are stronger convergence results than the more typical $H^\sigma(\mathbb{T}^d)$ ($\sigma>\frac{d}{2}$) -convergence analysis, which is restricted to an analysis in smooth Sobolev spaces and to periodic boundary conditions.
We state our results in the next subsection.

\subsection{Result}\label{sec:results}

\begin{thm}[$\Omega = \mathbb{T}^d$]\label{GlobalT}
Let $T>0$, $d\le 3$, and $u_0\in H^{\alpha}(\mathbb{T}^d)$ with $\alpha \in [1+\frac{d}{4}, 3]$. Let $u\in \mathcal{C}([0,T], H^{\alpha}(\mathbb{T}^d))$ be the unique solution of \eqref{NLS}.
Then there exist $\tau_{\min} > 0$ depending on $T$ and on $\|u_0 \|_{H^\alpha}$, and $C_T $ a positive function depending on $T$ and $\sup_{[0,T]}\|u(t)\|_{H^\alpha}$, such that for every time step size $\tau \le \tau_{\min}$ the numerical solution $u^n$ given in equation \eqref{num} has the following error bound:
\be\label{Errbd}
\|u(n\tau) - u^n\|_{L^2} \le C_T(\sup_{[0,T]}\|u(t)\|_{H^\alpha}) \tau^{1+\gamma}, \quad 0 \le n\tau \le T,
\ee
for $\alpha$ and $\gamma\in [0,1]$ which satisfy
\be\label{condition}
\begin{cases}
\alpha > 1 + \frac{d}{2}\ \text{and} \ 0\le \gamma \le \frac{\alpha -1}{2}, \\
\alpha = 1 + \frac{d}{2} \ \text{and} \ 0 \le \gamma < \frac{d}{4}, \\
\alpha < 1 + \frac{d}{2} \ \text{and} \ 0\le\gamma \le \alpha -1 -\frac{d}{4}.
\end{cases}
\ee

\end{thm}
We now consider the case where $\Omega$ is a smooth bounded domain. Given that in this case the space $X^s$ (see Section \ref{sec:notation}) in which the solution belongs depends not only on Sobolev regularity but also on compatibility conditions which the solution must satisfy on the boundary, we divide the statement of our results depending on the compatibility conditions imposed on $u|_{\partial \Omega}$, (and on the order of convergence).

\begin{thm}[$\Omega \subset \mathbb{R}^d$ smooth bounded domain]\label{GlobalOmega} Let $\Omega$ be a smooth bounded domain of $\mathbb{R}^d$. We consider the NLS equation \eqref{NLS}, equipped with homogeneous Dirichlet boundary conditions. Given any $T>0$, and $d\le 3$, there exists $\tau_{min}> 0 $ depending on $T$ and the norm of the initial data such that we have the following:
\begin{enumerate}
\item  Given any $u_0\in (H^{1+d/4}\cap H^1_0)(\Omega)$ we have first-order convergence of the symmetric scheme \eqref{num},
$$
\|u^n-u(n\tau)\|_{L^2} \le C_T \tau,
$$
for all $\tau\le \tau_{\min}$, and $0 \le n\tau \le T$.
\item More generally, given any $u_0 \in (H^{\alpha}\cap H^1_0)(\Omega)$ with $\alpha\in [1+\frac{d}{4},2]$ we have the fractional convergence estimates \eqref{Errbd} for $\alpha$ and $\gamma$ which satisfies \eqref{condition}. In particular, we have
$$
\|u(n\tau) - u^n\|_{L^2} \le C_T 
\begin{cases}
\tau^{1+ \frac{\alpha -1}{2}} \ \text{if} \ 1 + \frac{d}{2} < \alpha \le 2, \\
\tau^{1+ \frac{d}{4} - \epsilon} \ \text{for}\ \alpha = 1 + \frac{d}{2}, \\
\tau^{ \alpha -d/4}, \ \text{for} \ 1+ \frac{d}{4} \le \alpha < 1 + \frac{d}{2},
\end{cases}
$$
for $0 \le n\tau \le T$, $\tau \le \tau_{min}$, and for any $\epsilon > 0$.

\item By allowing for more compatibility condition on the boundary we have the following second-order convergence result for an initial data $u_0 \in X^3 = \{ u \in H^3(\Omega): u|_{\partial \Omega} = 0, \Delta u|_{\partial \Omega}=0 \ \text{in} \ L^2(\partial \Omega)\}$,
$$
\|u^n-u(n\tau)\|_{L^2} \le C_T \tau^2,
$$
{for all $\tau\le \tau_{\min}$, and $0 \le n\tau \le T$.}
\end{enumerate}

\end{thm}

We start by making a few remarks on Theorem \ref{GlobalT}, set on the torus $\mathbb{T}^d$. 
Let us first mention that the symmetric low-regularity integrator \eqref{num} requires less regularity assumptions than classical symmetric schemes (see \cite{fracSplit, L, Besse, Peterseim}). Indeed, for example the authors \cite{fracSplit} 
require $H^2$ solutions to obtain first order convergence of a Lie splitting scheme for NLS, and the author \cite{L} requires $H^4$-solutions for second-order convergence of a Strang splitting method, whereas we require $H^{1+\frac{d}{4}}$ and $H^3$ to obtain first and resp. second order convergence.

We compare this result to previous convergence results of explicit low-regularity integrators for the NLS equation \eqref{NLS}, which are not symmetric and hence do not have good structure preservation properties (see Figures \ref{fig:mass},~\ref{fig:energy}). 
To the best of our knowledge, this is the first fractional convergence results of a low-regularity schemes to be obtained from first to second order. 
We compare our full first and second-order convergence result with the work of \cite{YGP}, which also obtains first order convergence in $L^2(\mathbb{T}^d)$ for solutions $u(t) \in H^{1+d/4}(\mathbb{T}^d)$.
For the second order convergence in $L^2$ of their asymmetric second-order low-regularity integrator the author \cite{YGP} asks for solutions $u(t)\in H^{2+d/4}$, whereas the symmetric low-regularity integrator \eqref{num} requires a bit more regularity, namely $H^3$ solutions. 
Moreover, convergence of order 
$\tau^{1+\gamma} $ in $H^r$-norm, $r>d/2$, for $u_0 \in H^{2\gamma + r + 1}(\mathbb{T}^d)$ easily follows from the proof of Theorem \ref{GlobalT}. This is to be compared with asymmetric resonance-based schemes which would typically ask for $u_0 \in H^{\gamma + r + 1}(\mathbb{T}^d)$, $r > \frac{d}{2}$. See \cite{OS1, OWY} for a first and resp. second order analysis. We refer to Remark~\ref{rem:H3sol} which discusses the necessity of requiring three additional derivatives on the solution to obtain second-order convergence of the scheme \eqref{num}.
Finally, we note that the above analysis works analogously when adding a potential term $uV$ to equation \eqref{NLS}. One would need to require the same regularity assumption (and boundary conditions) on $V$ as is required on $u$ in the above theorem. This follows exactly as done in \cite{YGP}. 

We finish by comparing our result with the work of \cite{VMS}, which introduced a symmetrized low-regularity integrator for the Schr\"odinger map (SM), where they relate the SM flow to the NLS equation set on 1-d torus $\mathbb{T}$ via the Hasimoto transform. 
The analysis of their scheme is however restricted to the 1-d torus, and to first order convergence in smooth Sobolev spaces $H^r(\mathbb{T})$, $r>1/2$.
The results we present here go beyond the more typical $H^r(\mathbb{T})$ error analysis ($r>\frac{1}{2}$), by pushing down the error analysis to $L^2$ for first and up to second order convergence. Furthermore, we do not restrict ourselves to Fourier-based techniques, and hence to periodic boundary conditions, as is testified by Theorem \ref{GlobalOmega}.  Using the techniques presented in this article, one can also obtain a symmetric low-regularity approximation to the Schr\"odinger map in a more general setting than \cite{VMS}.

We now comment upon Theorem \ref{GlobalOmega}. To our knowledge, this is the first convergence result which goes beyond the first-order convergence analysis of a low-regularity integrator when set on a smooth bounded domain $\Omega \subset \mathbb{R}^d$. 
We refer to \cite[Corollary 20]{RS} where the authors show first order convergence in $L^2(\Omega)$ of the asymmetric low-regularity scheme \eqref{LRexpl}
while analogously asking for $(H^{1+\frac{d}{4}} \cap H^1_0)(\Omega)$ solutions.

We also compare our result 
to the work of \cite{Peterseim} which introduces a
mass and energy conserving variant of the Crank-Nicolson method as its time-discretization. They show first order convergence on a smooth bounded domain $\Omega \subset \mathbb{R}^d$ under -among other assumptions- $u_{t} \in L^2(0,T;H^2(\Omega))$, and obtain second order convergence under -among other assumptions- for $u_{tt} \in L^2(0,T;H^2(\Omega))$, while assuming $u \in C([0,T], H^2(\Omega))$ throughout their analysis. 
In contrast to the above classical results Theorem \ref{GlobalOmega} permits less regularity assumptions on $u(t)$, namely less than $H^2$-solutions for first order, and less than $H^4$-solutions for second order.

We also mention that for the 1-d NLS equation with Neumann boundary conditions a low-regularity integrator has been introduced by \cite{lowNeumann}, where using harmonic analysis techniques they could prove up to almost first order convergence with $H^1$-data.

\subsection{Outline of the paper}
In Section \ref{sec:notation} we set the scene and introduce the spaces and norms, together with crucial nonlinear estimates, which we will work with throughout the error analysis section. In Section \ref{sec:impl} we analyze the implicit nature of the scheme; we show that it is well-defined and establish a crucial a priori estimate on the numerical solution. Finally, in Section \ref{sec:errAnal} we prove the fractional global error estimates presented in Theorems \ref{GlobalT} and \ref{GlobalOmega}. First, in Section \ref{sec:local} the fractional local error bounds are obtained, followed by Section \ref{sec:stab} where the stability estimate is shown, and from which the convergence results then naturally follow.

\section{Norms, spaces, and nonlinear estimates}\label{sec:notation}
The norm and space used during the error analysis will depend on the domain $\Omega$ and boundary conditions imposed. We will treat the case where $\Omega = \mathbb{T}^d$ with periodic boundary conditions, and the case of homogeneous Dirichlet boundary conditions when placed on a smooth bounded domain of $\mathbb{R}^d$. In the case where $\Omega = \mathbb{T}^d$ the domain of the operator $\mathcal{L} = -\Delta$ is $D(\mathcal{L}) = H^2(\mathbb{T}^d)$, whereas for Dirichlet boundary conditions we have that $D(\mathcal{L}) = (H^2\cap H^1_0)(\Omega)$.
 We can define powers of $\mathcal{L}$, $\mathcal{L}^s$, for $s \ge 0$ using the spectral resolution, 
 and define the space $X^s(\Omega) = \mathcal{D}(\mathcal{L}^{s/2})$ as the domain of the operator $\mathcal{L}^{s/2}$, where $X^0(\Omega) = L^2(\Omega)$. We define the norm on $X^s(\Omega)$ by the usual graph norm
$$
\|u\|_{s}^2 = \|u\|^2 + \|\mathcal{L}^{s/2}u \|^2, \quad s \ge 0,
$$
where $\|u\| = \|u\|_{L^2}$ is the $L^2(\Omega)$-norm. We will be interested in characterizing the space $X^s(\Omega)$ depending on the domain $\Omega$ at study.

\subsection{The case of periodic boundary conditions} 
In the case of periodic boundary conditions we have that
$$X^{s}(\mathbb{T}^d) = H^{s}(\mathbb{T}^d) := \left\{u = \sum_{k \in \mathbb{Z}^d}u_k \frac{e^{ikx}}{\sqrt{ (2\pi)^d}} \in L^2(\mathbb{T}^d): |u|_s^2 \triangleq \sum_{k \in \mathbb{Z}^d} |k|^{2 s}|u_k|^2 < \infty \right\}$$
with equivalence of norms
%
$$
\|u\|_{s}^2 = \|u\|_{L^2(\mathbb{T}^d)}^2 + \|(-\Delta)^{s/2}u\|_{L^2(\mathbb{T}^d)}^2
= \sum_{k \in \mathbb{Z}^d} (1+ |k|^{2s})|u_k|^2
= \|u\|_{H^s}^2,
$$
%
where $\displaystyle{ u_k = \frac{1}{\sqrt{ (2\pi)^d}}\int_{\mathbb{T}^d} u e^{-ikx}dx} $.

\subsection{The case of Dirichlet boundary conditions}

We will be interested in characterizing the domain $X^s(\Omega)$ for $s \in [0,2] \cup \mathbb{N}$ (see Theorem \ref{GlobalOmega}). 

In the case where $s = m \in \mathbb{N}$ we have the following characterization (see \cite[Lemma 3.1]{VT})
$$
X^m = \{u \in H^m(\Omega): \Delta^j u = 0 \ \text{in} \ L^2(\partial \Omega) \ \text{for} \ j < m/2 \},
$$
with equivalence of the norms on $H^m(\Omega)$ and $X^m$ for functions in $X^m$. 

To treat the case where $s$ is not an integer we first introduce the following fractional Sobolev-type spaces known as the {\it Sobolev-Slobodetskij}, {\it Gagliardo} or {\it Aronszajn} space. Given any $s>0$ of the form $s = m+\sigma$, with $m\in\mathbb{N}$ and $\sigma \in (0,1)$, we define
$$
H^s(\Omega) = \{u\in H^m(\Omega): D^\alpha u \in H^\sigma(\Omega) \ \text{for any} \ \alpha \ \text{s.t.}\ |\alpha|=m \},$$
endowed with the norm
$$
 \|u\|_{H^s}^2 = \sum_{|\alpha|=0}^{m} \|D^\alpha u\|^2 + \sum_{|\alpha|=m} |D^\alpha u|_{H^{\sigma}}^2.
$$
For $s=m$ an integer the space $H^s(\Omega)$ coincides with the usual Sobolev space $H^m(\Omega)$, and for $\sigma \in (0,1)$ we have
$$
H^\sigma(\Omega) = \left\{u\in L^2(\Omega): |u|_{H^\sigma}^2 := \int_{\Omega} \int_{\Omega} \frac{|u(x) - u(y)|^2}{|x-y|^{^{d+2\sigma}}}dxdy < \infty \right\}.
$$
We note that all of the fractional Sobolev spaces which we introduce here can also be defined by using interpolation theory. Indeed, the above space is an intermediary Banach space between $L^2(\Omega)$ and $H^1(\Omega)$, and can be defined by interpolation as
$$
H^\sigma(\Omega) = [L^2(\Omega), H^1(\Omega)]_{\sigma},
$$
see \cite[Appendix 1]{Vaz} and \cite{atsuchi}.
Finally, for $s \in (1/2, 2]$ we define
$$
H^s_D(\Omega) = \{u\in H^s(\Omega): u|_{\partial \Omega}=0 \ \textit{in} \ L^2(\partial\Omega) \},
$$
it follows from the above that $D(\mathcal{L}) = H^2_D(\Omega)$. We can now express $X^s$ in terms of Sobolev spaces for $s \in [0,2]\backslash{\frac{1}{2}}$ (see \cite[Theorem 16.12]{atsuchi}),
$$
X^s(\Omega) = 
\begin{cases}
H^s(\Omega) \quad \text{if} \ 0\le s < \frac{1}{2}\\
H^s_D(\Omega) \quad \text{if} \ \frac{1}{2} < s \le 2
\end{cases},
$$
with norm equivalence
\be\label{equivnorms}
C^{-1} ||u||_{H^s} \le \|u\|_{s} \le C||u||_{H^s}, \quad u \in X^s,
\ee
for some constant $C>0$.
In the special case where $s=1/2$ we have that $X^{1/2} = H^{1/2}_{00}$ is the intermediate space defined by
$$
H^{1/2}_{00}(\Omega) := \{u\in H^{1/2}(\Omega): |u|_{H^{1/2}_{00}}^2 := \int_{\Omega}
 \frac{u^2(x)}{\text{dist}(x,\partial \Omega)}dx < \infty \},
$$
with equivalence of norms on $X^{1/2}$ as in \eqref{equivnorms}, see \cite[Prop 2.2]{Antil}.

\subsubsection{Bilinear and nonlinear estimates}

In this section we introduce bilinear estimates that are fundamental for the global error analysis, which we now motivate. 
The results we present in this article go beyond the more typical $H^s$ error analysis ($s>\frac{d}{2}$), by pushing down the analysis to $L^2$ and obtaining fractional rates of convergence, from first up to second order. In particular, to obtain these fractional rates when $\gamma < d/4$, we need to work in the low-order Sobolev spaces $H^{2\gamma}$ (see Section \ref{sec:errAnal}). In order to obtain sharp low-regularity error estimates in theses spaces we call upon three bilinear estimates (see equations \eqref{estim1}, \eqref{estim2}, and \eqref{estim3} below) which are taylored to require the least regularity assumptions on $u$ when bounding the local error terms (see also Remark \ref{rem:bilin}).

Let $\gamma \ge 0$ and $\epsilon > 0$.
Throughout the error analysis we will use the following bilinear estimates, depending on the values of $\gamma$. 
In the regime $\gamma > d/4$ we call upon the classical bilinear estimate
\be\label{bilinsmooth}
\|uv\|_{H^{2\gamma}} \lesssim 
\|u\|_{H^{2\gamma}} \| v\|_{H^{2\gamma}}, \ \text{for} \ \gamma > \frac{d}{4},
\ee
whereas in the regime $\gamma \in [0,d/4)$ we exploit the following three bilinear estimates,
\be\label{estim1}
\|uv\| \lesssim 
\|u\|_{H^{\frac{d}{4} + \gamma}} \| v\|_{H^{\frac{d}{4}-\gamma}} \ \text{for} \ 0\le \gamma < \frac{d}{4},
\ee
\be\label{estim2}
\|uv\|_{H^{2\gamma}} \lesssim 
\|u\|_{H^{\frac{d}{4} + \gamma}} \| v\|_{H^{\frac{d}{4}+\gamma}} \ \text{for} \ 0\le \gamma < \frac{d}{4}, 
\ee
and
\be\label{estim3}
\|uv\|_{H^{2\gamma}} \lesssim 
\|u\|_{H^{\frac{d}{2} + \epsilon}} \| v\|_{H^{2\gamma}} \ \text{for} \ 0\le \gamma \le \frac{d}{4},
\ee
for any $\epsilon > 0 $.
The above estimates are particular cases of \cite[Theorem 8.3.1]{hormander}, valid either on $\mathbb{R}^d$ or $\mathbb{T}^d$. Furthermore, 
for a smooth bounded domain $\Omega \subset \mathbb{R}^d$, Stein's extension theorem (\cite[p.154]{adams}) guarantees the existence of a total extension operator, bounded both from $L^2(\Omega)$ to $L^2(\mathbb{R}^d)$ and from $H^m(\Omega)$ to $H^m(\mathbb{R}^d)$, for any $m\in \mathbb{N}$. By interpolation, this operator is bounded from $H^s(\Omega)$ to $H^s(\mathbb{R}^d)$ for any $s\le m$ (see \cite[p. 208]{adams}). The estimates \eqref{bilinsmooth}, \eqref{estim1}, \eqref{estim2} and \eqref{estim3} consequently hold on $\Omega$ by extending $u$ and $v$ to $\mathbb{R}^d$, applying the estimates on their extensions, and restricting their product to $\Omega$.

\begin{rem}[Bilinear estimates in low-order Sobolev spaces $H^{2\gamma}$, $\gamma < d/4$]\label{rem:bilin}
A natural bilinear estimate which is essential for  an analysis in the spaces $H^{2\gamma}$, $\gamma < d/4$, is the estimate \eqref{estim3}. This estimate is an analogue of the estimate \eqref{bilinsmooth} in the smooth case where $\gamma > d/4$. These two estimates allow to start and fall back on the same space $H^{2\gamma}$. 
However, the estimate \eqref{estim3} requires more regularity on $u$ than on $v$ ($\frac{d}{2} + \epsilon > 2\gamma$, for $\gamma < d/4$), and asks for $H^{\frac{d}{2}+\epsilon}$-regularity on the solution. One can obtain more optimal bounds which require less regularity assumptions by equally distributing the regularity on $u$ and on $v$. Indeed, by assuming that $u$ and $v$ have the same regularity, applying the estimate \eqref{estim2} requires $\frac{d}{4} + \gamma$ additional derivatives, which is better than $\frac{d}{2}+\epsilon$ for $\gamma < d/4$. While the estimate \eqref{estim1} is used when $v$ requires $2\gamma$ derivatives more than $u$, and balances the regularity requirement to again ask for $\frac{d}{4} + \gamma$ on both $u$ and $v$ (see Proposition \ref{prop:R1bd}).
\end{rem}

We now consider the nonlinearity, which we denote by 
\be\label{nonlin}
f(u,\bar{u})(t,x) = -iu^2(t,x)\bar{u}(t,x).
\ee
One can easily deduce from the inequalities \eqref{bilinsmooth} and \eqref{estim3} together with the equivalence of norms on $X^s$ the following estimates on the nonlinearity \eqref{nonlin}
\be\label{A2.2}
\begin{array}{c}
\|f(w,\bar{w})\| _{s} 
\le c_{s,\sigma} \|w\|_{\sigma}^2\|w\|_{s}  \le C_{s,\sigma}(\|w\|_{{\sigma}}) \|w\|_{s} \\\\
\|f(v, \bar{v}) - f(w,\bar{w})\|_{s} 
\le c_{s,\sigma}\|v-w \|_{s}\sum_{k=0}^2\|v \|^k_{\sigma} \| w\|^{2-k}_{\sigma}
 \le C_{s,\sigma}(\|v\|_{{\sigma}}, \|w\|_{{\sigma}})\|v-w \|_{s}
\end{array},
\ee
where $\sigma = \frac{d}{2} + \epsilon$, $c_{s,\sigma} > 0$, and $C_{s,\sigma}(\|u\|,\|v\|)$ denotes a generic constant which depends continuously on the bounded arguments $\|u\|$ and $\|v\|$. In the regime $s > \frac{d}{2}$ the above holds with $\sigma = s$.

\begin{rem}[An analysis for very rough solutions]
The main ingredient throughout the error analysis section of this article rests upon the crucial bilinear estimates given above, and restricts the solution to belong to the Sobolev space $H^s$, $s>d/2$. In order to consider very rough solutions $u \in H^s$, $s\le d/2$ one needs to call upon more refined tools such as discrete Bourgain spaces when working on the torus (\cite{ORS2}), and discrete Strichartz estimates when working on the full space (\cite{ORS1}). This delicate error-analysis is out of scope for this paper.
\end{rem}

Lastly, as we are interested in obtaining fractional error estimates we will call upon the following estimate several times throughout the error analysis section.
For $\gamma \in [0,1]$, we have,
\be\label{exp}
\left\lVert \frac{(e^{it\Delta }-1)}{(- t\Delta)^\gamma} u\right\lVert \le 2^{1-\gamma}\|u\|.
\ee
The above estimate easily follows from the usual bound
$$
|e^{ix} -1| \le 2^{1-\gamma} |x|^\gamma, \quad \gamma \in [0,1],
$$
and using the discrete spectral decomposition of the operator $\mathcal{L} =-\Delta$.

We finish this section by stating the definition of a commutator-type term, which is used in order to obtain low-regularity error estimates (see Section \ref{sec:local}). For $H(v_1, \cdots, v_n), \ n \ge 1$, a function and $L$ a linear operator, we define the commutator-type term $\mathcal{C}[H,L]$ as
$$
\mathcal{C}[H,L](v_1, \cdots, v_n) = -L(H(v_1, \cdots, v_n)) + \sum_{j=1}^{n} \partial_{v_j} H(v_1, \cdots, v_n) \cdot Lv_j.
$$
In our setting, $H(v_1, v_2) =f(v_1, v_2)= -iv_1^2 v_2$ is the nonlinearity given in \eqref{nonlin}, $L = i\Delta$, and hence
\be\label{deriv}
 \partial_{v_1} f(v_1, v_2) = -2iv_1v_2 \quad \text{and}  \quad \partial_{v_2}f(v_1,v_2) = -iv_1^2
\ee
and
\be\label{com}
\mathcal{C}[f,i\Delta](v_1,v_2) = -\Delta(v_1^2v_2) + 2v_1v_2\Delta v_1 + v_1^2 \Delta v_2=-2( |\nabla v_1|^2 v_2 + 2v_1 \nabla v_1 \cdot \nabla v_2).
\ee

\section{The implicit nature of the scheme}\label{sec:impl}

In this section we deal with the question of solving the nonlinear equation \eqref{num} at a given time step. We also provide an a priori bound on the numerical solution $\varphi^\tau(v)$ in terms of $v$, which is crucial for the convergence analysis. 
We recall from Figure \ref{CPU} that the implicit nature of the scheme does not adversely affect the computational cost of the method.
In the following, we fix $v \in X^\sigma$ for some $\sigma > d/2$. We note that $v$ will play the role of the element $u^n$ in the scheme \eqref{num}. We then introduce the map
$$
z \mapsto S(z) = e^{i\tau\Delta}v + \psi^{\tau/2}_{E}(v) + \psi^{\tau/2}_{I}(z),
$$
and wish to prove that it admits a unique fixed point given by $\varphi^\tau(v)$.

We start by introducing some useful estimates on the map $S$. 
\begin{prop}\label{prop:S} Given any $\sigma> d/2$ we have
$$
\|S(z_1) - S(z_2)\|_{\sigma} \le \tau M(\|z_1\|_{\sigma}, \|z_2\|_{\sigma}) \|z_1-z_2\|_{\sigma}, \quad \text{and} \quad \|S(e^{i\tau\Delta}z_1) - e^{i\tau\Delta}z_1\|_{\sigma} \le \tau \widetilde{M}(\|z_1\|_{\sigma}),
$$
where $M(\|z_1\|, \|z_2\|)$ and $\widetilde{M}(\|z_1\|)$ denote generic constants which depend continuously on their arguments $\|z_1\|$ and $\|z_2\|$.
\end{prop}
\begin{proof}[Proof of Proposition \ref{prop:S}]
The proof follows directly from the definition of the map $S$, the scheme \eqref{num} and of the estimate \eqref{A2.2}.
\end{proof}

The following theorem shows that the implicit scheme \eqref{num} is well-defined, and admits an a priori bound.
\begin{thm}\label{thm:fpt} Let $R>0$ and $\sigma > d/2$. There exists $\tau_R>0$ such that, for all $\tau\le \tau_R$ and $v \in X^\sigma$ with $\| v\|_{\sigma} \le R$, we have that $\varphi^\tau(v)$ defined in \eqref{num} is given by
\be\label{fpt}
\varphi^\tau(v) \overset{H^\sigma}{= \ } \lim_{j\rightarrow +\infty} S^j(e^{i\tau\Delta}v).
\ee
Moreover, under the same conditions, we have
\be\label{implBd}
\|\varphi^\tau(v)\|_{\sigma}\le 2R.
\ee
\end{thm}

\begin{proof}[Proof of Theorem \ref{thm:fpt}]
For notational convenience we let $x_j = S^j(x_0)$, with $x_0 = e^{i\tau\Delta}v$.
We first show by induction that for sufficiently small $\tau$ we have the bound
\be\label{fptbd}
\|x_j\|_{\sigma} \le 2R, \quad j\ge 0. 
\ee
We choose $\tau_R>0$ such that $\tau_R M(2R,2R) \le 1/2$ and $\tau_R\widetilde{M}(R)\le R/2$, with $M$ and $\widetilde M$ from Proposition \ref{prop:S}. We assume that $ \|x_j\|_{\sigma} \le 2R$, $\forall j \le J$. It follows that for $\tau \le \tau_R$ we have,
\begin{align*}
\|x_{J+1} -x_0\|_{\sigma} &\le \sum_{j=1}^{J} ||S(x_j) - S(x_{j-1})||_{\sigma} + \|S(x_0) - x_0 \|_{\sigma}\\
&\le \sum_{j=0}^{J}  \left( \prod_{k=1}^{j} \tau M(\|x_k \|_{\sigma}, \|x_{k-1} \|_{\sigma}) \right)  \|S(x_0) - x_0\|_{\sigma} \\
&\le \frac{R}{2} \sum_{j=0}^{J} \frac{1}{2^j} \\
&\le R.
\end{align*}
By recalling that, by assumption, $\| x_0\| = \| v\| \le R$, we conclude from the above that 
$$
\|x_{J+1}\|_{\sigma} \le 2R,
$$
and hence by induction bound \eqref{fptbd} holds.

It then follows that for all $ \tau \le \tau_R$, $(x_j)_{j\in\mathbb{N}}$ is a Cauchy sequence. Indeed, for $m>p$ we have
\begin{align*}
\| x_m - x_p \|_{\sigma} \le \sum_{j=p}^{m-1} \|S(x_j) - S(x_{j-1}) \|_{\sigma} \le \frac{R}{2}\sum_{j=p}^\infty \frac{1}{2^j} \underset{p\rightarrow \infty}{\longrightarrow} 0.
\end{align*}
This implies that the sequence $(x_j)_{j\in\mathbb{N}}$ converges in $H^\sigma$ to the unique fixed-point $\varphi^\tau(v)$ of $S$, and the characterization \eqref{fpt} follows. Finally, by passing to the limit in \eqref{fptbd} we obtain the desired a priori bound \eqref{implBd} on $\varphi^\tau(v)$, which concludes the proof.
\end{proof}

\section{Error Analysis}\label{sec:errAnal}
In this section we will prove the following proposition.
\begin{prop}\label{prop:toProve}
Let $T>0$, and $\gamma \in [0,1]\backslash \{ \frac{d}{4}\}$. 
Then there exists $\tau_{\min} > 0$ such that for every time step $\tau \le \tau_{\min}$ the numerical solution $u^n$ given in equation \eqref{num} has the following error bound: 
\be
\|u(n\tau) - u^n\|_{L^2} \le C_T(\sup_{[0,T]}\|u(t)\|_{H^\alpha}) \tau^{1+\gamma}, \quad 0 \le n\tau \le T,
\ee
where $\alpha$ is given by
\be\label{condition}
\begin{cases}
\alpha = 2\gamma +1 \quad &\text{if} \ \frac{d}{4} < \gamma \le 1 \\
\alpha = \gamma +1 + \frac{d}{4} \quad &\text{if} \ \ 0 \le \gamma < \frac{d}{4}
\end{cases},
\ee

and where $\tau_{\min}$ depends on $T$ and on $\| u_0\|_{H^\alpha}$, and $C_T$ is a positive function of its argument, depending on $T$. 
\end{prop}

Given a fixed convergence rate this proposition expresses the regularity assumptions needed in order to obtain this rate, while on the other hand given a fixed regularity assumption on the initial data Theorems \ref{GlobalT} and \ref{GlobalOmega} express the convergence rates one can attain with the method \eqref{num}. We now link these results.

\begin{proof}[Proof of Theorem \ref{GlobalT} and Theorem \ref{GlobalOmega}] By writing the convergence rate $\gamma$ in terms of the regularity assumptions needed on the solution,
it directly follows from the above proposition that the convergence rate \eqref{Errbd} holds for $\alpha$ and $\gamma$ which satisfy:
\be\label{bob1}
\begin{cases}
1+ \frac{d}{2} < \alpha \le 3, \quad \gamma = \frac{\alpha -1}{2}\\
1+ \frac{d}{4} \le \alpha < 1 + \frac{d}{2},\quad \gamma = \alpha -1 -\frac{d}{4}.
\end{cases}
\ee
The proof of Theorem \ref{GlobalT} then follows by using the fact that 
if \eqref{Errbd} holds for some $\tilde\alpha \in [1+ \frac{d}{4},3], \tilde\gamma \in [0,1]$, then the error bound also holds for any $\alpha \ge \tilde\alpha$, and $\gamma\le \tilde\gamma$. In particular, we recover the case $\alpha = 1 + \frac{d}{2}$ in Theorem \ref{GlobalT} by applying the second line in \eqref{bob1} with $(\tilde\alpha, \tilde\gamma) = (1 + \frac{d}{2} - \epsilon,\frac{d}{4} - \epsilon)$ to obtain convergence for $0 \le \gamma < \frac{d}{4}$. See Figure \ref{cv-graph} which illustrates graphically the convergence result.

\begin{figure}[H]
\centering
\begin{tikzpicture}[scale=0.9]
\draw[draw=none, fill=blue, opacity=0.3] (0,1.2)--(2,2)--(3,3)--(0,3)--cycle;
\draw[->,ultra thick] (0,0)--(4,0) node[right]{$\gamma$};
\draw[->,ultra thick] (0,0)--(0,4) node[above]{$\alpha$};
\node[label=below:{$d/4$}] at (2,0) {$\bullet$};
\node[label=left:{$1+d/2$}] at (0,2) {$\bullet$};
\node[label=left:{$1+d/4$}] at (0,1.2) {$\bullet$};
\node[label=left:{$3$}] at (0,3) {$\bullet$};
\node[label=below:{$1$}] at (3,0) {$\bullet$};
\node[label=below:{$0$}] at (0,0) {$\bullet$};
\draw (0,1.2)--(2,2);
\draw (2,2)--(3,3);
\draw [fill=white] (2,2) circle[radius= 0.2 em]; 
\end{tikzpicture}
\caption{Illustration of the interplay between the regularity parameter $\alpha$ and the convergence rate parameter $\gamma$ in the convergence result stated in Proposition \eqref{prop:toProve} and Theorem \ref{GlobalT}. We plot the regularity assumption ($u(t) \in H^{\alpha}$) needed in order to obtain convergence of order $\tau^{1+\gamma}$.
}\label{cv-graph}
\end{figure}
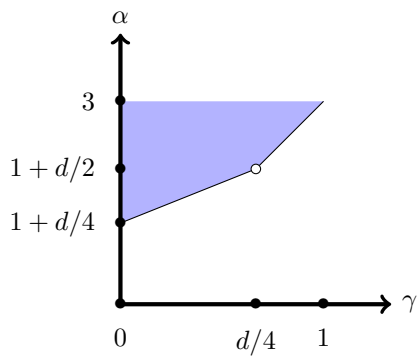

The proof of Theorem \ref{GlobalOmega} follows in the same manner, with the added constraint of the boundary conditions. Namely, we require $u \in C([0,T], X^s)$ and $s = 2\gamma +1$ or $\gamma+1+d/4$, where the boundary conditions are imposed in the definition of the space $X^s$. See Section \ref{sec:notation} for the definition of the spaces $X^s$.

\end{proof}

In order to prove Proposition \ref{prop:toProve} we combine local error bounds together with a stability argument to conclude via a Lady Windermere's fan argument. We start by showing the local error bound of order $\tau^{2+\gamma}$ with the regularity assumptions stated in Proposition \ref{prop:toProve}.
\subsection{Local error analysis}\label{sec:local}
We decompose the local error term as follows,
\begin{align*}
u(t_n + \tau) - \varphi^{\tau}\bigl(u(t_n)\bigr)  &= \int_0^\tau e^{i(\tau-s)\Delta}f\bigl(u(t_n + s), \bar{u}(t_n + s)\bigr) ds - \psi_{E}^{\tau/2}\bigl(u(t_n)\bigr)-\psi_{I}^{\tau/2}\Bigl(\varphi^{\tau}\bigl(u(t_n)\bigr)\Bigr)
\\
&=\mathcal{R}_1(\tau,t_n) + \mathcal{R}_2(\tau,t_n) + \mathcal{R}_3(\tau,t_n),
\end{align*}
with 
\begin{align}\label{R^2_1}
\mathcal{R}_1(\tau,t_n) &= \left(\int_0^{\tau/2} e^{i(\tau-s)\Delta}f\bigl(e^{is\Delta}u(t_n), e^{-is\Delta}\bar{u}(t_n)\bigr)ds - \psi_{E}^{\tau/2}\bigl(u(t_n)\bigr)\right) \\\nonumber
&\quad + \left( \int_{\tau/2}^\tau e^{i(\tau-s)\Delta}f\bigl(e^{is\Delta}u(t_n),e^{-is\Delta}\bar{u}(t_n)\bigr)ds - \psi_{I}^{\tau/2}\bigl(e^{i\tau\Delta} u(t_n)\bigr) \right), \\\nonumber
\end{align}
\begin{align}\label{R_2}
\mathcal{R}_2(\tau,t_n) &=  \psi_{I}^{\tau/2}\left(e^{i\tau\Delta} u(t_n)\right) - \psi_{I}^{\tau/2}\Bigl(\varphi_{\tau}\bigl(u(t_n)\bigr)\Bigr) ,
\end{align}
and
\begin{align*}
\mathcal{R}_3(\tau,t_n) &= \int_0^\tau e^{i(\tau-s)\Delta}\Bigl(f\bigl(u(t_n + s), \bar{u}(t_n + s)\bigr) - f\bigl(e^{is\Delta}u(t_n), e^{-is\Delta}\bar{u}(t_n)\bigr)\Bigr)ds.
\end{align*}

We start by estimating the first error term $\mathcal{R}_1(\tau,t_n)$. This term is the one which asks for the most regularity, and hence dictates the regularity assumptions required on the solution, and thereby the initial data. We then proceed by estimating each of the terms $\mathcal{R}_2(\tau,t_n)$ and $\mathcal{R}_3(\tau,t_n)$, to obtain a cancelation in their sum thanks to the symmetry of the scheme, yielding the desired local error estimate.

\begin{prop}\label{prop:R1bd} The error term $\mathcal{R}_1(\tau,t_n) $ satisfies the following bound,
\be\label{localToProve}
\|\mathcal{R}_1(\tau,t_n) \| \le  \left\{
\begin{array}{ll}
C_{T}\left(\sup_{t\in [0,T]}\|u(t) \|_{X^{2\gamma +1}}\right) \tau^{2+ \gamma} &\quad \text{if} \ \gamma > \frac{d}{4}\vspace{1mm}\\ 
C_{T}\left(\sup_{t\in [0,T]}\|u(t) \|_{X^{\gamma +1 + \frac{d}{4}}}\right) \tau^{2+ \gamma} & \quad \text{if} \ 0 \le \gamma < \frac{d}{4} 
\end{array} \right. ,
\ee
for $0 \le t_n \le T$.
\end{prop}

\begin{proof}[Proof of Proposition \ref{prop:R1bd}]
 We define the filtered function as
\be\label{filterfcn}
\begin{split}
\mathcal{N}(\tau, s, \zeta, \Delta, v) 
&= -ie^{i(\tau-s)\Delta} [ (e^{is\Delta} v)^2(e^{is\Delta}e^{-2i\zeta \Delta} \bar{v})]\\
& = e^{i(\tau-s)\Delta}  f(e^{is\Delta} v, e^{i(s-2\zeta)\Delta}\bar{v})
\end{split}
\ee
which plays a fundamental role in the derivation and analysis of our scheme on general domains.
In the above expression we duplicate the time variable into $s$ and $\zeta $, pulling out a factor $e^{i s\Delta}$ in front of the conjugate term $e^{i(s-2\zeta)\Delta}\bar{v}$. Taylor expanding in the variable $s$ yields the right cancellation with the factor $e^{i(\tau-s)\Delta}$ to recover, after integrating in the variable $\zeta$, the explicit term $\psi_{E}^{\tau/2}(v)$ in the scheme \eqref{num}, as is detailed below. Similar filtering techniques are used in \cite{RS, YGP, ABBSBdd}. 

Let $v = u(t_n)$. By Taylor expanding the filtering function \eqref{filterfcn} around $s=0$, we obtain that the first term in \eqref{R^2_1} satisfies
\begin{align}\label{R1e1}
\int_0^{\tau/2} e^{i(\tau-\zeta)\Delta}f(e^{i\zeta\Delta}v, e^{-i\zeta\Delta}\bar{v})d\zeta &= \int_0^{\tau/2} \mathcal{N}(\tau, \zeta, \zeta, \Delta, v)d\zeta \\\nonumber
& = \int_0^{\tau/2}\mathcal{N}(\tau, 0, \zeta, \Delta, v)d\zeta + \int_0^{\tau/2} \int_0^\zeta \partial_s \mathcal{N}(\tau ,s,\zeta, \Delta, v) ds d\zeta
\\\nonumber
&= \psi_{E}^{\tau/2}(v)+ \int_0^{\tau/2} \int_0^\zeta e^{i(\tau-s)\Delta} \mathcal{C}[f,i\Delta](e^{is\Delta}v, e^{i(s-2\zeta)\Delta}\bar{v}) ds d\zeta,
\end{align}
where to obtain the last line we used the definition of the $\varphi_1$ function (see \eqref{LRexpl}) to obtain that
$$
\int_0^{\tau/2} e^{-2i\zeta\Delta}d\zeta = \frac{\tau}{2}\frac{e^{-i\tau\Delta} - 1}{-i\tau \Delta}=\frac{\tau}{2} \varphi_1(-i\tau \Delta)
$$
and the definition of commutator term $\mathcal{C}[f,i\Delta](u,v)$ given in \eqref{com}.

Similarly, using the filtering function \eqref{filterfcn}, we can treat the second line in \eqref{R^2_1} by Taylor expanding around $s=\tau$. This yields

\be
\begin{aligned}\label{R1e2}
\int_{\tau/2}^\tau e^{i(\tau-\zeta)\Delta}f(e^{i\zeta\Delta}v, e^{-i\zeta\Delta}\bar{v})d\zeta &= \int_{\tau/2}^{\tau} \mathcal{N}(\tau, \zeta, \zeta, \Delta, v)d\zeta \\
& = \int_{\tau/2}^\tau\mathcal{N}(\tau, \tau, \zeta, \Delta, v)d\zeta - \int_{\tau/2}^\tau\int_{\zeta}^{\tau}  \partial_s \mathcal{N}(\tau ,s,\zeta, \Delta, v) ds d\zeta
\\
&= \psi_{I}^{\tau/2}(e^{i\tau\Delta}v)  - \int_{\tau/2}^\tau \int_\zeta^{\tau} e^{i(\tau-s)\Delta} \mathcal{C}[f,i\Delta](e^{is\Delta}v, e^{i(s-2\zeta)\Delta}\bar{v}) ds d\zeta,
\end{aligned}
\ee
where to go from the second line in the above to the third we used that
\begin{align*}
\int_{\tau/2}^\tau\mathcal{N}(\tau, \tau, \zeta, \Delta, v)d\zeta & = -i(e^{i\tau\Delta}v)^2 \left( \left(\int_{\tau/2}^\tau e^{-2i\zeta\Delta}d\zeta\right)e^{i\tau\Delta}\bar{v}\right) \\ \nonumber
&= -i(e^{i\tau\Delta}v)^2 \left(  \left(\int_0^{\tau/2}e^{-2i(\tau-\zeta)\Delta}d\zeta\right) e^{i\tau\Delta}\bar{v}\right) \\\nonumber
&= -i\frac{\tau}{2} (e^{i\tau\Delta}v)^2\varphi_1(i\tau\Delta)(e^{-i\tau\Delta}\bar{v}),\\\nonumber
&= \psi_{I}^{\tau/2}(e^{i\tau\Delta}v).
\end{align*}

By definition of $\mathcal{R}_1(\tau,t_n)$ and by using \eqref{R1e1} and \eqref{R1e2} we have that

\be\label{R^1_2bd}
\begin{aligned}
\mathcal{R}_1(\tau,t_n) &= \int_0^{\tau/2} \int_0^\zeta e^{i(\tau-s)\Delta} \mathcal{C}[f,i\Delta](e^{is\Delta}v, e^{is\Delta}e^{-2i\zeta\Delta}\bar{v}) ds d\zeta \\
&\quad - \int_{\tau/2}^\tau \int_\zeta^{\tau} e^{i(\tau-s)\Delta} \mathcal{C}[f,i\Delta](e^{is\Delta}v, e^{is\Delta}e^{-2i\zeta\Delta}\bar{v}) ds d\zeta\\
&= \int_0^{\tau/2} \int_0^\zeta e^{i(\tau-s)\Delta} \mathcal{C}[f,i\Delta](e^{is\Delta}v, e^{is\Delta}e^{-2i\zeta\Delta}\bar{v}) \\
&\quad - e^{is\Delta} \mathcal{C}[f,i\Delta](e^{i(\tau-s)\Delta}v, e^{i(\tau-s)\Delta}e^{-2i(\tau-\zeta)\Delta}\bar{v})dsd\zeta.
\end{aligned}
\ee

Note that the local error structure lead by the commutator-type terms in the above expression requires less regularity assumptions than what is required by classical methods, such as exponential integrators of splitting methods (see \cite{ExpInt, L}). 
Indeed, from the explicit form \eqref{com} of the commutator term we see that this error term requires only one additional derivative on the initial datum rather than two (see also \cite{YGP, ABBSBdd}).

We now show that thanks to the symmetry of the scheme, we obtain a cancelation in the second-order error term \eqref{R^1_2bd} yielding (up to) a third-order remainder.
For notational convenience we let $w_1(s) = e^{is\Delta}v,$ $w_2(s,\zeta) = e^{i(s-2\zeta)\Delta}\bar{v}$, $z_1(s) = e^{i(\tau-s)\Delta}v$, and $z_2(s,\zeta) =e^{i(2\zeta - \tau-s)\Delta} \bar{v}$ and we denote the integrand by
\be\label{R_r}
R_r( \tau, s, \zeta, v) = e^{i(\tau-s)\Delta} \mathcal{C}[f,i\Delta](w_1(s), w_2(s,\zeta)) - e^{is\Delta} \mathcal{C}[f,i\Delta](z_1(s), z_2(s,\zeta)).
\ee
It follows from the above equations \eqref{R^1_2bd} and \eqref{R_r} that
\begin{align*}
\|\mathcal{R}_1(\tau,t_n)\| &\le \frac{\tau^2}4 \sup_{s,\zeta \in [0,\tau/2]} \|R_r( \tau, s, \zeta,  v) \|.
\end{align*}
It remains to show that
 \begin{align}\label{R1rbound}
 \sup_{s,\zeta \in [0,\tau/2]}\|R_r( \tau, s, \zeta, v)\| \le 
 \begin{cases}
C(\|v\|_{2\gamma + 1})\tau^\gamma\ \text {if} \ \gamma >d/4 \\ 
C(\|v\|_{\gamma + 1 + \frac{d}{4}})\tau^\gamma \ \text {if} \ \gamma < d/4
\end{cases}.
 \end{align}
We first approximate the exponentials appearing in front of both commutator terms in \eqref{R_r} to obtain the following first approximation result on $R_r$.
\begin{lem}\label{lem:R1_1} We have
\begin{align*}
R_r( \tau, s, \zeta, v) = R^1_r( \tau, s, \zeta, v) + R^2_r( \tau, s, \zeta, v),
\end{align*}
with
$$
R^1_r(\tau,  s, \zeta, v) = (e^{i(\tau-s)\Delta}-1) \mathcal{C}[f,i\Delta](w_1(s),w_2(s,\zeta)) - (e^{is\Delta}-1)\mathcal{C}[f,i\Delta](z_1(s), z_2(s,\zeta))
$$ and
$$ R^2_r( \tau, s, \zeta, v) = \mathcal{C}[f,i\Delta](w_1(s), w_2(s,\zeta)) - \mathcal{C}[f,i\Delta](z_1(s), z_2(s,\zeta)),
$$
 which satisfy
\be\label{R1r_bound}
\sup_{s,\zeta \in [0,\tau/2]}\|R^i_r( \tau, s, \zeta, \tau, v)\| \le
\begin{cases}
C(\|v\|_{2\gamma + 1})\tau^\gamma\ \text {if} \ \gamma >d/4, \\
C(\|v\|_{\gamma + 1 + \frac{d}{4}})\tau^\gamma \ \text {if} \ \gamma < d/4,
\end{cases}
\ee
for $i=1,2$.

\end{lem}
\begin{proof}[Proof of Lemma \ref{lem:R1_1}]
We write the error term $R^1_r$ as
\be\label{R^1_r}
\begin{aligned}
R^1_r( s, \zeta, \tau, v, \bar{v}) = & (\tau-s)^{\gamma}\left(\frac{e^{i(\tau-s)\Delta}-1}{(\tau-s)^\gamma(-\Delta)^\gamma}\right) (-\Delta)^\gamma \mathcal{C}[f,i\Delta](w_1(s),w_2(s,\zeta)) \\
&- s^\gamma\left(\frac{e^{is\Delta}-1}{s^\gamma(-\Delta)^\gamma}\right)(-\Delta)^\gamma\mathcal{C}[f,i\Delta](z_1(s), z_2(s,\zeta)).
\end{aligned}
\ee
Using the bound given in equation \eqref{exp}, and the boundedness of $e^{it\Delta}$ on Sobolev spaces, it follows from equation \eqref{R^1_r}  that we are left to provide a bound on $ \mathcal{C}[f,i\Delta](v, \bar{v})$ of the form
\be\label{bdcom}
\|\mathcal{C}[f,i\Delta](v, \bar{v})\|_{{2\gamma}} \le 
\begin{cases}
C(\|v\|_{2\gamma + 1})\ \text {if} \ \gamma >d/4 \\
C(\|v\|_{\gamma + 1 + \frac{d}{4}}) \ \text {if} \ \gamma < d/4
\end{cases}.
\ee

%

From the definition of the commutator \eqref{com}, and by using the equivalence of norms $\|\cdot\|_{s}$ and $\|\cdot \|_{H^{s}}$ on $X^{s}$ (see Section \ref{sec:notation}), it follows that for $d/4 < \gamma \le 1$ 
\begin{align*}
\||\nabla v | ^2 \bar{v} \|_{{2\gamma}} &\lesssim \||\nabla v | ^2 \bar{v} \|_{H^{2\gamma}}\\
&\lesssim \|v\|_{H^{2\gamma}} \|\nabla v \cdot \nabla v \|_{H^{2\gamma}} \\
&\lesssim \|v\|_{H^{2\gamma}}\|v\|_{H^{2\gamma +1}}^2 \\
&\lesssim C(\|v\|_{2\gamma + 1}),
\end{align*}
where we used the estimate \eqref{bilinsmooth}. 
Similarly, in the case $0\le\gamma < d/4$ we have
\begin{align*}
\||\nabla v | ^2 \bar{v} \|_{{2\gamma}} &\lesssim \|v\|_{H^{d/2 + \epsilon}} \| \nabla v \cdot \nabla v\|_{H^{2\gamma}} \\
&\lesssim \|v\|_{H^{d/2 + \epsilon}}\|v\|_{H^{\gamma + 1 + \frac{d}{4}}}^2 \\
& \lesssim C(\|v\|_{\gamma + 1 + \frac{d}{4}}), 
\end{align*}
where to obtain the first line we used estimate \eqref{estim3}, to go from the first to the second line we used the estimate \eqref{estim2} and concluded using the fact that $d/2 + \epsilon < \gamma + 1 + \frac{d}{4}$ together with the equivalence of norms on $X^{\gamma + 1 + \frac{d}{4}}$.

We can bound the second term in the commutator-type term \eqref{com} in the same manner to obtain the desired bound \eqref{bdcom}.
The estimate \eqref{R1r_bound} on $R^1_r$ follows immediately.

We now deal with the approximation of the second error term $R^2_r$ by approximating each of the exponentials appearing in the arguments of the commutator terms, namely on $w_1, w_2, z_1 $ and $z_2$. Given the form of the commutator \eqref{com} and of $w_1,w_2,z_1,z_2$, each term to approximate will either be of the form
\be\label{err1}
\left\lVert w \nabla ((e^{i\xi\Delta} -1) u) \cdot \nabla z \right\lVert
\ee
or 
\be\label{err2}
\left\lVert \left((e^{i\xi\Delta} -1) w\right) \nabla u \cdot \nabla z \right\lVert
\ee
for $\xi \in [0,\tau]$, and where we use the boundedness of $e^{it\Delta}$ on Sobolev spaces ($t\in\mathbb{R}$).
We can approximate \eqref{err1} as follows, given any $\epsilon >0$,
\begin{align*}
\left\lVert w \nabla ((e^{i\xi\Delta} -1) u) \cdot \nabla z \right\lVert
&\lesssim \xi^\gamma
\begin{cases}
 \|w\|_{H^{2\gamma}} \left\lVert \nabla (-\Delta)^\gamma \left(\frac{(e^{i\xi\Delta} -1)}{ (-\xi\Delta)^{\gamma}} u\right) \right\lVert \|\nabla z\|_{H^{2\gamma}} \ \text {if} \ \gamma >d/4\vspace{1mm} \\ 
\|w\|_{H^{\frac{d}{2} + \epsilon}} \left\lVert \nabla (-\Delta)^\gamma \left(\frac{(e^{i\xi\Delta} -1)}{(-\xi\Delta)^{\gamma}} u\right) \right\lVert_{H^{\frac{d}{4} -\gamma}} \|\nabla z \|_{H^{\frac{d}{4} + \gamma}} \ \text {if} \ \gamma <d/4
\end{cases}\vspace{1mm} \\
&\le\xi^\gamma
\begin{cases}
 C(\|w\|_{{2\gamma}}, \|u \|_{{2\gamma + 1 }}, \|z\|_{{2\gamma + 1 }})  \ \text {if} \ \gamma >d/4 \vspace{1mm}\\
C(\|w\|_{\frac{d}2 + \epsilon}, \|u \|_{{\gamma + 1 + \frac{d}{4}}}, \|z \|_{{\gamma + 1 + \frac{d}{4}}}) \ \text {if} \ \gamma <d/4
\end{cases},
\end{align*}
where we used the Sobolev embedding $H^{\sigma}  \hookrightarrow L^\infty$, for $\sigma > \frac{d}{2}$, and the estimate \eqref{estim1} to obtain the first inequality. To obtain the second line in the above we used the equivalence of norms, thanks to the fact that $u,v,w$ belong to $ X^{2\gamma +1}$ or $X^{\gamma + 1 +d/4}$ respectively, as well as the estimate \eqref{exp}. Hence, given that $\frac{d}{2} + \epsilon <  \gamma + 1 +\frac{d}{4} $, the term above satisfies the desired bound of the form \eqref{R1r_bound}.

Furthermore, for the expression \eqref{err2} we have,
\begin{align*}
\left\lVert \left((e^{i\xi\Delta} -1) w\right) \nabla u \cdot \nabla z \right\lVert
& \lesssim \xi^\gamma 
\begin{cases}
\left\lVert(-\Delta)^\gamma\frac{(e^{i\xi\Delta} -1)}{(-\xi\Delta)^\gamma} w\right\lVert \|\nabla u \cdot \nabla z\|_{H^{2\gamma}}  \ \text {if} \ \gamma >d/4 \\ \\
\left\lVert(-\Delta)^\gamma \frac{(e^{i\xi\Delta} -1)}{ (-\xi\Delta)^\gamma} w\right\lVert_{H^{1+\frac{d}{4}-\gamma}}\|\nabla u \cdot \nabla z\|_{H^{2\gamma}}  \ \text {if} \ \gamma <d/4
\end{cases} \\
& \lesssim \xi^\gamma 
\begin{cases}
\left\lVert(-\Delta)^\gamma\frac{(e^{i\xi\Delta} -1)}{(-\xi\Delta)^\gamma} w\right\lVert \|\nabla u \|_{H^{2\gamma}} \| \nabla z\|_{H^{2\gamma}}  \ \text {if} \ \gamma >d/4 \\ \\
\left\lVert(-\Delta)^\gamma \frac{(e^{i\xi\Delta} -1)}{ (-\xi\Delta)^\gamma} w\right\lVert_{H^{1+\frac{d}{4}-\gamma}}\|\nabla u \|_{H^{\frac{d}{4} + \gamma}} \|\nabla z\|_{H^{\frac{d}{4} + \gamma}}  \ \text {if} \ \gamma <d/4
\end{cases} \\
&\le \xi^\gamma
\begin{cases}
C(\|w\|_{{2\gamma}}, \|u \|_{{2\gamma + 1 }}, \|z\|_{{2\gamma + 1 }})  \ \text {if} \ \gamma >d/4 \\
C(\|w\|_{{\gamma + 1 + \frac{d}{4}}}, \|u \|_{{\gamma + 1 + \frac{d}{4}}}, \|z \|_{{\gamma + 1 + \frac{d}{4}}}) \ \text {if} \ \gamma <d/4,
\end{cases}
\end{align*}
where to obtain the first line we used the Sobolev embedding $H^\sigma  \hookrightarrow L^{\infty}$, $\sigma>\frac{d}{2}$, and the estimate $\|uz\| \lesssim \|u\|_{H^{1+\frac{d}{4}-\gamma}} \| z\|_{H^{2\gamma}} $, $\gamma \in [0,1]$, (see \cite[ Theorem 8.3.1]{hormander}). In order to obtain the second line in the above we again used the estimates \eqref{bilinsmooth} and \eqref{estim2}, and to obtain the third line we used the equivalence of norms on the spaces $X^s$ together with the estimate \eqref{exp}.

By approximating each of the exponentials in the commutator terms defining $R^2_r$ and by collecting the error terms which are either of the form \eqref{err1} or \eqref{err2} we recuperate the desired $\tau^\gamma$ bound in \eqref{R1r_bound}.

\end{proof}

We can conclude from Lemma \ref{lem:R1_1} that we have the desired bound \eqref{R1rbound}, and hence the bound \eqref{localToProve} on $\mathcal{R}_1(\tau,t_n)$. 
\end{proof}
We now continue with the bound on the two remaining terms $(\mathcal{R}_2 + \mathcal{R}_3)(\tau,t_n)$, which asks for less regularity assumptions than the boundedness of $\mathcal{R}_1(\tau,t_n)$, as shown below.


\begin{prop}\label{prop:R23}
For $\gamma\in [0,1] $, we have the following fractional bound,
\be\label{R23}
\left\lVert(\mathcal{R}_2+ \mathcal{R}_3)(\tau,t_n)\right\lVert \le C\left(\sup_{[0,T]}\| u(t)\|_{{2\gamma}},\sup_{[0,T]}\| u(t)\|_{{\sigma}}\right)\tau^{2+\gamma},
\ee
given any $\sigma > \frac{d}{2}$. In particular we have the bound,
\be\label{localToProve23}
\left\lVert(\mathcal{R}_2 +\mathcal{R}_3) (\tau,t_n) \right\lVert \le
\begin{cases}
C_{T,\gamma}\left(\sup_{t\in [0,T]}\|u(t) \|_{{2\gamma +1}}\right) \tau^{2+ \gamma}, \ \ \text{if} \ \gamma > \frac{d}{4}\vspace{1mm}\\
C_{T,\gamma}\left(\sup_{t\in [0,T]}\|u(t) \|_{{\gamma +1 + \frac{d}{4}}}\right) \tau^{2+ \gamma},  \ \ \text{if} \ 0 \le \gamma < \frac{d}{4} \\ 
\end{cases},
\ee
for $0 \le t_n \le T$.
\end{prop}

\begin{proof}[Proof of Proposition \ref{prop:R23}]
First, we rewrite the error term $\mathcal{R}_3(\tau,t_n)$ by making suitable Taylor expansions on~$f$. We start by expanding $u(t_n+ \zeta)$ locally up to second order :
\be\label{smoothExp}
u(t_n + \zeta) = e^{i\zeta \Delta}u(t_n) + \zeta f^n + \tilde{R}(\zeta, t_n)
\ee
where $f^n = f(u(t_n),\bar{u}(t_n))$ and 
\be\label{Rtilde}
\tilde{R}(\zeta, t_n) = \int_0^\zeta e^{i(\zeta-s)\Delta} f(u(t_n+s), \bar{u}(t_n+s))ds - \zeta f^n.
\ee
Using the above expansion for $u$ we rewrite the error term as
\begin{align}\label{E^1}
\mathcal{R}_3(\tau, t_n) = \int_0^\tau e^{i(\tau-\zeta) \Delta}&\Bigg( f\left(e^{i\zeta \Delta}u(t_n) + \zeta f^n + \tilde{R}(\zeta, t_n), \right.\Big.\left.\Big. e^{-i\zeta \Delta}\bar{u}(t_n) + \zeta \overline{f^n} + \overline{\tilde{R}(\zeta, t_n)}\right) \Big.\\ \Big.\nonumber
&  - f\left(e^{i\zeta\Delta}u(t_n), e^{-i\zeta\Delta}\bar{u}(t_n)\right) \Bigg) d\zeta.
\end{align}
For notational convenience we let $a_1 := e^{i\zeta \Delta}u(t_n)+ \zeta f^n$.
By Taylor expanding $f$ around $(a_1, \bar{a}_1)$ and $(e^{i\zeta\Delta} u(t_n),e^{-i\zeta\Delta} \bar{u}(t_n))$ respectively we obtain,
\begin{align}\label{Taylor12}
f\left(a_1 + \tilde{R}(\zeta, t_n), \bar{a}_1 + \overline{\tilde{R}(\zeta, t_n)}\right) = f\left(a_1, \bar{a}_1\right) + E_{1}(\zeta) \\ \nonumber
f\left(a_1, \bar{a}_1\right) = f\left(e^{i\zeta\Delta} u(t_n), e^{-i\zeta\Delta} \bar{u}(t_n)\right) + E_{2}(\zeta)
\end{align}
where
\begin{align}\label{E1}
&E_1(\zeta) = \int_0^1 \partial_{v_1}f\left(a_1 + \theta \tilde{R}(\zeta,t_n), \bar{a}_1 + \theta\overline{\tilde{R}(\zeta,t_n)} \right)\cdot \tilde{R}(\zeta,t_n) \\ \nonumber
&\qquad\qquad\qquad+ \partial_{v_2} f\left( a_1 + \theta \tilde{R}(\zeta,t_n), \bar{a}_1 + \theta\overline{\tilde{R}(\zeta,t_n)} \right)\cdot \overline{\tilde{R}(\zeta,t_n)}d\theta \\\label{E2}
&E_2(\zeta) = \zeta\int_0^1 [ \partial_{v_1}f\left(e^{i\zeta\Delta} u(t_n)+ \theta\zeta f^n,e^{-i\zeta\Delta} \bar{u}(t_n)+\theta\zeta\overline{f^n}\right)\cdot f^n \\ \nonumber
& \qquad\qquad\qquad +  \partial_{v_2}f\left(e^{i\zeta\Delta} u(t_n)+ \theta\zeta f^n,e^{-i\zeta\Delta} \bar{u}(t_n)+\theta\zeta\overline{f^n}\right)\cdot \overline{f^n} ]d\theta,
\end{align}
and $\partial_{v_1}f$ together with $\partial_{v_2}f$ are given in \eqref{deriv}.
Hence, plugging the above into equation \eqref{E^1} yields,
\begin{align}\label{E1E2}
\mathcal{R}_3(\tau, t_n) =&
 \int_0^\tau e^{i(\tau - \zeta)\Delta} E_{1}(\zeta) d\zeta
+ \int_0^\tau e^{i(\tau-\zeta) \Delta}E_2(\zeta) d\zeta\\ \nonumber
=& \mathcal{E}_1(\tau,t_n) + \mathcal{E}_2(\tau,t_n).
\end{align}

We first deal with the term in the decomposition above which is of highest order and hence is the simplest to bound, namely the third order term $\mathcal{E}_1(\tau,t_n)$.
In view of obtaining the bound \eqref{R23} on $(\mathcal{R}_2 + \mathcal{R}_3)(\tau,t_n)$, we first show that $\mathcal{E}_1(\tau,t_n)$ satisfies this bound.

\begin{lem}\label{lem:E1} We have the following fractional bound on $\mathcal{E}_1(\tau,t_n)$,
$$
\|\mathcal{E}_1(\tau,t_n)\| \le C(\sup_{[0,T]}\| u(t)\|_{{\sigma}}, \sup_{[0,T]}\| u(t)\|_{{2\gamma}})\tau^{2+\gamma},
$$
for any $\gamma \in [0,1]$, and $\sigma > d/2$.
\end{lem}
\begin{proof}
It follows from \eqref{E1} that in order to obtain the above bound we need to show that

$$
\|E_{1}(\zeta) \| \le C (\sup_{[0,T]}\|u(t)\|_{{\sigma}}, \sup_{[0,T]}\|u(t)\|_{2\gamma})\zeta^{1+\gamma}.
$$
By equation \eqref{E1} and by using the Sobolev embedding $H^\sigma \hookrightarrow L^\infty$ we have that for all $\sigma >\frac{d}{2}$,
\be\label{E1.1_norm}
\begin{split}
\|E_{1}(\zeta)\| 
&\le \sup_{\theta \in ]0,1[} \biggl( \|{ \partial_{v_1}}f \left(a_1 + \theta \tilde{R}(\zeta,t_n), \bar{a}_1 + \theta\overline{\tilde{R}(\zeta,t_n)} \right)\|_{\sigma}\\
&\qquad\qquad+ \|{ \partial_{v_2}}f \left(a_1 + \theta \tilde{R}(\zeta,t_n), \bar{a}_1 + \theta\overline{\tilde{R}(\zeta,t_n)} \right)\|_{\sigma} \biggr) 
\|\tilde{R}(\zeta, t_n) \| \\
&\le C\biggl( \sup_{t \in [0,T]} \|u(t)\|_{\sigma}, \sup_{(\zeta, t)\in [0,\tau]\times [0,T] } \|\tilde{R}(\zeta, t)\|_{\sigma}\biggr)
\|\tilde{R}(\zeta,t_n) \|,
\end{split}
\ee
where the last inequality follows by using the explicit form of the derivatives \eqref{deriv}, the bilinear inequality \eqref{bilinsmooth}, the first estimate of equation \eqref{A2.2}, and the fact that $e^{i\zeta\Delta}$ is an isometry on Sobolev spaces.
Next, we show that
\be\label{E2}
\sup_{(\zeta,t)\in  [0,\tau]\times[0,T]} \|\tilde{R}(\zeta,t) \|_{\sigma} < \infty, \quad \text{and} \quad
 \|\tilde{R}(\zeta,t_n) \| \le C_T(\sup_{ [0,T]}\|u(t)\|_{{\sigma}}, \sup_{ [0,T]}\|u(t)\|_{{2\gamma}}) \zeta^{1+\gamma}.
\ee
We obtain the first bound by using the first estimate of equation \eqref{A2.2} on $f$ with $r = \sigma$,
$$
\sup_{(\zeta, t)\in  [0,\tau]\times[0,T]}\|\tilde{R}(\zeta,t) \|_{\sigma} \le \tau C(\sup_{t \in [0,T]}\|u(t)\|_{\sigma}) < +\infty.
$$
Next, we obtain the second fractional estimate of equation \eqref{E2} by making the following decomposition,
\begin{align}\label{Rtilde2}
\tilde{{R}}(\zeta,t_n) &= \tilde R_1(\zeta,t_n) + \tilde R_2(\zeta,t_n),
\end{align}
with
$$
\tilde R_1(\zeta,t_n) =  
\int_0^\zeta (\zeta - s)^\gamma \frac{(e^{i(\zeta-s)\Delta} -1)}{(-(\zeta - s)\Delta)^\gamma} (-\Delta)^\gamma f(u(t_n+s), \bar{u}(t_n+s))ds$$ 
and
$$
\tilde R_2(\zeta,t_n) = \int_0^\zeta f(u(t_n+s), \overline{u}(t_n+s) )ds - \zeta f^n.
$$
Using the fractional bound \eqref{exp} and the nonlinear estimate \eqref{A2.2} we have that $\tilde R_1(\zeta,t_n)$ is bounded by
\begin{align*}
\| \tilde R_1(\zeta,t_n) \|
&\le \zeta^{1+\gamma} C(\sup_{t\in [0,T]}\|u(t)\|_{{\sigma}}, \sup_{t\in [0,T]}\|u(t)\|_{{2\gamma}}).
\end{align*}

Next, by iterating Duhamel's formula in the first term of $\tilde{R}_2(\zeta,t_n)$ we obtain the following expansion for $\tilde{R}_2(\zeta,t_n)$, 
\begin{align}\label{Rtilde3}
\tilde{R}_2(\zeta,t_n)
= \tilde{R}_{2,1}(\zeta, t_n) + \tilde{R}_{2,2}(\zeta, t_n),
\end{align}
where 
$$
\tilde{R}_{2,1}(\zeta, t_n) = \int_0^{\zeta} f(e^{is\Delta}u(t_n), e^{-is\Delta}\bar{u}(t_n))ds - \zeta f^n,
$$
\begin{align*}
\tilde{R}_{2,2}(\zeta, t_n) &=  \int_0^\zeta 
\int_0^1 { \partial_{v_1}}f\left(e^{is\Delta}u(t_n) + \theta \tilde{R}_{2,2}^r(s), e^{-is\Delta}\bar{u}(t_n) + \theta\overline{\tilde{R}_{2,2}^r(\zeta)} \right)\cdot \tilde{R}_{2,2}^r(s) \\ 
&\qquad\qquad\qquad+ { \partial_{v_2}} f\left( e^{is\Delta}u(t_n) + \theta \tilde{R}_{2,2}^r(s), e^{-is\Delta}\bar{u}(t_n) + \theta\overline{\tilde{R}_{2,2}^r(s)} \right)\cdot \overline{\tilde{R}_{2,2}^r(s)}d\theta ds,
\end{align*}
and $\tilde{R}_{2,2}^r(s) = \int_0^s e^{i(s-s_1)\Delta}f(u(t_n+s_1), \bar{u}(t_n+s_1))ds_1$. 

Using the nonlinear estimate \eqref{A2.2}, one easily obtains the bound 
$$
\|\tilde{R}_{2,2}(\zeta, t_n)\| \le C(\sup_{[0,T]}\|u(t)\|_{{\sigma}})\zeta^2, 
$$
and hence in particular the $\zeta ^{1+\gamma}$ bound for $\gamma \in[0,1]$. 

In order to deal with the first term in the decomposition \eqref{Rtilde3}, we Taylor expand the exponentials appearing in $\tilde{R}_{2,1}(\zeta, t_n)$ which yields,
\begin{align*}
\tilde{R}_{2,1}(\zeta, t_n) &=  \int_0^\zeta 
\int_0^1 { \partial_{v_1}}f\left(u(t_n) + \theta (e^{is\Delta}-1)u(t_n), \bar{u}(t_n) + \theta(e^{-is\Delta}-1)\bar{u}(t_n)\right)\cdot \tilde{R}_{2,1}^r(s) \\ 
&\qquad\qquad\qquad+ { \partial_{v_2}} f\left( u(t_n) + \theta (e^{is\Delta}-1)u(t_n), \bar{u}(t_n) + \theta(e^{-is\Delta}-1)\bar{u}(t_n) \right)\cdot \overline{\tilde{R}_{2,1}^r(s)}d\theta ds,
\end{align*}
with $\tilde{R}_{2,1}^r(s) = s^\gamma  \frac{(e^{is\Delta} -1)}{ (-s\Delta)^\gamma} (-\Delta)^\gamma u(t_n) $.
Using the fractional estimate \eqref{exp} we obtain the bound 
$$\|\tilde{R}_{2,1}^r(s)\| \le C_T(\sup_{[0,T]} \|u(t)\|_{{2\gamma}})s^{\gamma}.$$ 
Therefore, by using the usual bilinear inequality \eqref{bilinsmooth} we achieve the desired bound on $\tilde{R}_{2,1}$;
$$\|\tilde{R}_{2,1}(\zeta) \|\le C(\sup_{[0,T]} \|u(t)\|_{\sigma}, \sup_{[0,T]} \|u(t)\|_{{2\gamma}})\zeta^{1+\gamma},
$$
which concludes the proof of Lemma \ref{lem:E1}.

\end{proof}


Now that we have dealt with the third order term $\mathcal{E}_1(t_n,\tau)$ in the decomposition \eqref{E1E2} of $\mathcal{R}_3(t_n,\tau)$, we are left to consider the term $\mathcal{E}_2(t_n,\tau)$, together with the term $\mathcal{R}_2(t_n,\tau)$ defined at equation \eqref{R_2}. 
First, we rewrite $\mathcal{E}_2(t_n,\tau)$ as a second order term with a third order remainder. The goal being that this second order term cancels with the second order part of the term $\mathcal{R}_2(t_n,\tau)$, thereby only leaving third order remainders.

By using the definition of ${\partial_{v_1}}f$ and ${\partial_{{v_2}}}f$ given in \eqref{deriv} we have that
\begin{align*}
E_2(\zeta) = -i\zeta \int_0^1 [2(e^{i\zeta \Delta}u(t_n) + \theta \zeta f^n)(e^{-i\zeta \Delta}\bar{u}(t_n) + \theta \zeta \overline{ f^n})f^n  + (e^{i\zeta \Delta}u(t_n) + \theta\zeta f^n)^2 \overline{f^n}] d\theta.
\end{align*}
We can separate the first order terms in the above with the higher order ones to obtain the following decomposition for $E_2(\zeta)$,
$$
E_2(\zeta) =  \tilde{E_2}(\zeta) + E_{2}^r(\zeta),
$$
with
$$
\tilde{E_2}(\zeta)= -i\zeta \left( 2(e^{i\zeta \Delta}u(t_n))(e^{-i\zeta \Delta}\bar{u}(t_n))f^n + (e^{i\zeta \Delta}u(t_n))^2 \overline{f^n}\right),
$$
and where
one can easily show using the nonlinear estimate \eqref{A2.2} with $r = \sigma$ that $E_{2}^r(\zeta)$ has the following bound: $\|E_{2}^r(\zeta)\| \le C_T(\sup_{[0,T]}\|u(t)\|_{\sigma})\zeta^2$.

We let
$$\tilde{\mathcal{E}}_2(\tau,t_n) = \int_0^\tau e^{i(\tau-\zeta) \Delta}\tilde{E}_2(\zeta) d\zeta \quad\text{and}\quad 
\tilde{\mathcal{E}}_2^r(\tau,t_n) = \int_0^\tau e^{i(\tau-\zeta) \Delta}\tilde{E}_2^r(\zeta) d\zeta,
$$
where the error term produced has the bound $\| \tilde{\mathcal{E}}_2^r(\tau,t_n)\| \le C(\sup_{[0,T]}\|u(t)\|_{\sigma})\tau^3$, which in particular satisfies the $\tau^{2+\gamma}$ bound for $\gamma \in [0,1]$.

It remains to show that the sum of the remaining terms to bound $(\mathcal{R}_2+\tilde{\mathcal{E}}_2 )(\tau,t_n)$ also satisfy the $\tau^{2+\gamma}$ bound given in equation \eqref{R23}. In view of this, one last approximation step is made on the term $\tilde{\mathcal{E}}_2 (\tau,t_n)$ before estimating its sum with the term $\mathcal{R}_2 (\tau,t_n)$.
By Taylor expanding around $\zeta = \tau$ the function
$\zeta \mapsto e^{i(\tau-\zeta)\Delta}(\tilde{E}_2(\zeta)/\zeta)$ 
we obtain the following approximation of $\tilde{\mathcal{E}}_2$,
$$
\tilde{\mathcal{E}}_2(\tau,t_n) = \tilde{\mathcal{E}}_{2,1}(\tau,t_n) + \tilde{\mathcal{E}}_{2,1}^r (\tau,t_n),
$$
with 
\be\label{E_21}
\tilde{\mathcal{E}}_{2,1}(\tau,t_n) = -i\frac{\tau^2}{2}\left( 2(e^{i\tau \Delta}u(t_n))(e^{-i\tau \Delta}\bar{u}(t_n))f^n + (e^{i\tau \Delta}u(t_n))^2 \overline{f^n} \right),
\ee
and where $\tilde{\mathcal{E}}_{2,1}^r (\tau)$ satisfies
$$
\|\tilde{\mathcal{E}}_{2,1}^r (\tau)\| \le C(\sup_{[0,T]}\|u(t)\|_{{\sigma}}, \sup_{[0,T]}\|u(t)\|_{{2\gamma}})\tau^{2+\gamma}.
$$
The above estimate follows from the definition of $\tilde E_{2}$, the estimate \eqref{exp}, and the expansion 
$$
e^{i\tau\Delta}v+ (e^{i\zeta\Delta} - e^{i\tau\Delta})v  = e^{i\tau\Delta}v + \left(\zeta^\gamma\frac{(e^{i\zeta\Delta} - 1)}{ (-\zeta\Delta)^\gamma} + \tau^\gamma\frac{(1 - e^{i\tau\Delta})}{ (-\tau\Delta)^\gamma}\right) (-\Delta)^\gamma v.
$$

The first estimate \eqref{R23} of Proposition \ref{prop:R23} then follows directly once the following lemma is established.
\begin{lem}\label{lem:R2E2}
The remaining error terms have the following bound,
\be\label{EstimR2E2}
\|\mathcal{R}_2(\tau, t_n) + \tilde{\mathcal{E}}_{2,1}(\tau,t_n)\|\le C(\sup_{[0,T]}\|u(t)\|_{{\sigma}}, \sup_{[0,T]}\|u(t)\|_{{2\gamma}})\tau^{2+\gamma},
\ee
for any $\sigma >d/2$.
\end{lem}
\begin{proof}
We perform a very similar analysis as was done on the term $\mathcal{R}_3(\tau,t_n)$ to the term $\mathcal{R}_2(\tau,t_n)$ to show that it can be decomposed as a second and third order term. We then conclude by showing that this second order part coincides with the second order term $-\tilde{\mathcal{E}}_{2,1}(\tau,t_n)$.

First, we expand $\varphi^\tau(u(t_n))$ as follows, 
\begin{align}\label{varphiExp}
\varphi^\tau(u(t_n)) 
&= e^{i\tau \Delta}u(t_n) + \tau f^n+ R^1(\tau, t_n),
\end{align}
where 
$$R^1(\tau, t_n) = \Psi^\tau(u(t_n), \varphi^\tau(u(t_n))) - \tau f^n,
$$
and $\Psi^\tau$ is the nonlinear part of the numerical scheme \eqref{num}.
We show the same bounds in $X^\sigma$ and $L^2$ given in equation \eqref{E2} on the error term ${R}^1(\tau,t_n)$. First, using the estimate \eqref{A2.2} with $r=\sigma$ we have that 
$$
\sup_{t\in[0,T]}\|R^1(\tau,t) \|_{\sigma} \le \tau C(\sup_{[0,T]} \|u(t)\|_{\sigma}, \sup_{[0,T]}\|\varphi^\tau(u(t)) \|_{\sigma}) \le C(\sup_{[0,T]} \|u(t)\|_{\sigma}) < \infty,
$$
where we use the bound 
\be\label{varphiBd}
\|\varphi^\tau(u(t))\|_{\sigma} \le 2\| u(t)\|_{\sigma},
\ee
which follows from equation \eqref{implBd}.
We note that in what follows we will always use the above a priori bound \eqref{varphiBd} when bounding the term $\varphi^\tau(u(t_n))$, and hence will not show its explicit dependance. 
In order to obtain the $L^2$-bound on $R^1(\tau,t_n)$ we use the following expansion
\begin{align}\label{Rloc2}\nonumber
R^1(\tau,t_n) 
&= - i \frac{\tau}{2} e^{i\tau\Delta} \left( (u(t_n))^2 \varphi_1(-i\tau\Delta) \overline{u(t_n)} \right) - i\frac{\tau}{2} \Bigg( \left(e^{i\tau\Delta}u(t_n)+\Psi^\tau(u(t_n), \varphi^\tau(u(t_n)))\right)^2 \Bigg.\\
&\Big. \qquad \varphi_1(i\tau \Delta)\left(e^{-i\tau\Delta}\overline{u(t_n)}+\overline{\Psi^\tau(u(t_n), \varphi^\tau(u(t_n)))} \right)\Bigg) -\tau f^n
\end{align}
where we used equation~\eqref{num} and simply inserted into the term $\Psi^\tau(u(t_n), \varphi^\tau(u(t_n)))$ the definition of the scheme
$$
\varphi^\tau(u(t_n))=e^{i\tau\Delta}u(t_n)+\Psi^\tau(u(t_n), \varphi^\tau(u(t_n))).
$$
By expanding about $e^{i\tau\Delta}u(t_n)$ the second term in \eqref{Rloc2}, and then by approximating the remaining exponentials using the usual fractional estimate \eqref{exp}, we obtain that the first-order terms cancel leaving the following bound on $R^1(\tau,t_n)$,
\begin{align*}
\|R^1(\tau,t_n)\| \le&   \tau^{1+\gamma} C(\sup_{[0,T]}\|u(t)\|_{2\gamma}, \sup_{[0,T]}\|u(t)\|_{\sigma})+ \tau C(\sup_{n\tau\le T} \|\Psi^\tau(u(t_n), \varphi^\tau(u(t_n)))\|_{\sigma})\\
\le& \tau^{1+\gamma} C (\sup_{[0,T]}\|u(t)\|_{2\gamma}, \sup_{[0,T]}\|u(t)\|_{\sigma}),
\end{align*}
where in order to obtain the second line we used the estimate
$$
\| \Psi^\tau(u(t_n), \varphi^\tau(u(t_n)))\| \le C(\sup_{[0,T]}\|u(t)\|_{\sigma})\tau.
$$
We make note that in order to approximate the $\varphi_1$ functions appearing in \eqref{Rloc2} we use the following expansion
\be\label{phi1approx}
\tau \varphi_1(i\tau\Delta)v = \int_0^\tau e^{is\Delta}ds v = \tau v + \int_0^\tau \frac{(e^{is\Delta}-1)}{ (-s\Delta)^\gamma} s^\gamma ds (-\Delta)^\gamma v.
\ee
We conclude from the above calculations that the error term $R^1(\tau,t_n)$ satisfies the $X^\sigma$ and $L^2$ bounds,
\be\label{R1bd}
\sup_{t\in [0,T]} \|R^1(\tau,t) \|_{\sigma} < \infty \quad \text{and} \quad
 \|R^1(\tau,t_n) \| \le C(\sup_{ [0,T]}\|u(t)\|_{{\sigma}}, \sup_{ [0,T]}\|u(t)\|_{{2\gamma}}) \tau^{1+\gamma}.
\ee

We now return to the definition \eqref{R_2} of $\mathcal{R}_2(\tau,t_n)$ and to the expansion \eqref{varphiExp}. By letting $b_1 = e^{i\tau\Delta} u(t_n) + \tau f^n$ we have
\be\label{Taylor34}
\begin{aligned}
\psi_{I}^{\tau/2}\left( b_1 + R^1(\tau,t_n)\right) &= \psi_{I}^{\tau/2}(b_1) + \mathcal{E}_3(\tau,t_n) \\
\psi_{I}^{\tau/2}\left(b_1\right) &= \psi_{I}^{\tau/2}(e^{i\tau\Delta}u(t_n)) + \mathcal{E}_4(\tau,t_n),
\end{aligned}
\ee
where using the estimates in equation \eqref{R1bd} and the definition of $\psi_{I}^{\tau/2}$ it follows that
$$
\|\mathcal{E}_3(\tau,t_n) \| \le \tau^{2+\gamma} C(\sup_{ [0,T]}\|u(t)\|_{{\sigma}}, \sup_{ [0,T]}\|u(t)\|_{{2\gamma}}).
$$
Furthermore, it follows from equation \eqref{Taylor34} that by isolating the second order terms with the higher order ones we have the following expansion for $\mathcal{E}_4(\tau,t_n)$,

\begin{align}\label{E4exp}
\mathcal{E}_4(\tau,t_n) = -i\frac{\tau^2}{2} &\Big(
 2f^n(e^{i\tau \Delta}u(t_n))  (\varphi_1(i\tau\Delta) (e^{-i\tau \Delta}\overline{u(t_n)})) \Big.\\\nonumber
&\Big. +(e^{i\tau \Delta}u(t_n))^2\varphi_1(i\tau\Delta)\overline{f^n}\Big) + \mathcal{E}_4^r(\tau,t_n),
\end{align}
where from a simple calculation one obtains that $\mathcal{E}_4^r(\tau,t_n)$ satisfies
$$
\|\mathcal{E}_4^r(\tau,t_n)\| \le C(\sup_{ [0,T]}\|u(t)\|_{{\sigma}})\tau^3 \le C_T \tau^{2+\gamma}.
$$

By approximating the $\varphi_1$ functions in \eqref{E4exp} following the expansion given in equation \eqref{phi1approx}, and by using once again the fractional estimate \eqref{exp} we conclude from the above equations together with definition \eqref{E_21} of $\tilde{\mathcal{E}}_{2,1}$ that the bound \eqref{EstimR2E2} is met. This concludes the proof of Lemma \ref{lem:R2E2}.
\end{proof}

The proof of the above lemma concludes the proof of the first estimate \eqref{R23} on $(\mathcal{R}_2 + \mathcal{R}_3)(\tau,t_n)$ of Proposition \ref{prop:R23}.

The second estimate \eqref{localToProve23} of Proposition \ref{prop:R23} follows directly from the first estimate \eqref{R23} by noticing that for some small $\epsilon >0$ and with $\sigma = d/2 + \epsilon$, we have that $2\gamma$ and $\sigma$ are smaller than $2\gamma + 1$ for $\gamma > d/4$, and are also smaller than $\gamma + 1+d/4$ for $d\le 3$ (and $\gamma \in [0,1]$).
\end{proof}

\begin{rem}
Another way of writing the local error terms is as follows,
\begin{align}\label{otherDecomp}
u(t_n + \tau) - \varphi^{\tau}\bigl(u(t_n)\bigr)  &= \mathcal{R}(\tau,t_n) + \widetilde{\mathcal{R}}(\tau,t_n),
\end{align}
with
\begin{align*}
\mathcal{R}(\tau,t_n) &= \int_0^{\tau/2} e^{i(\tau-s)\Delta}\Bigl(f\bigl(u(t_n + s), \bar{u}(t_n + s)\bigr) - f\bigl(e^{is\Delta}u(t_n), e^{-is\Delta}\bar{u}(t_n)\bigr)\Bigr)ds \\
&+  \int_{\tau/2}^\tau e^{i(\tau-s)\Delta}\Bigl(f\bigl(u(t_n + s), \bar{u}(t_n + s)) - f\bigl(e^{i(s-\tau)\Delta}u(t_{n+1}), e^{-i(s-\tau)\Delta}\bar{u}(t_{n+1})\bigr)\Bigl)ds
\end{align*}
and
\begin{align*}
\widetilde{\mathcal{R}}(\tau,t_n) &= \left(\int_0^{\tau/2} e^{i(\tau-s)\Delta}f\bigl(e^{is\Delta}u(t_n), e^{-is\Delta}\bar{u}(t_n)\bigr)ds - \psi_{E}^{\tau/2}\bigl(u(t_n)\bigr)\right) \\
&\quad + \left( \int_{\tau/2}^\tau e^{i(\tau-s)\Delta}f\bigl(e^{i(s-\tau)\Delta}u(t_{n+1}),e^{-i(s-\tau)\Delta}\bar{u}(t_{n+1})\bigr)ds - \psi_{I}^{\tau/2}\Bigl(\varphi^{\tau}\bigl(u(t_n)\bigr)\Bigr) \right).
\end{align*}
 The above error decomposition uses the fact that on $[0,\frac{\tau}{2}]$ we center the approximation at the left-end point and on $(\frac{\tau}{2},\tau]$ at the right-end point. Hence, on each interval respectively we iterate the Duhamel expansions
\begin{align*}
&u(t_n + s) = e^{is\Delta}u(t_n) + \int_0^s e^{i(s-s_1)\Delta}f\bigl(u(t_n + s_1), \bar{u}(t_n + s_1)\bigr)ds_1, \quad s \in [0,\frac{\tau}{2}], \\
&u(t_n + s) = e^{i(s-\tau)\Delta}u(t_{n+1}) - \int_s^\tau e^{i(s-s_1)\Delta}f\bigl(u(t_n + s_1), \bar{u}(t_n + s_1)\bigr)ds_1, \quad s\in (\frac{\tau}{2}, \tau].
\end{align*}
Using the tools in Section~\ref{sec:local} one can bound the local error terms \eqref{otherDecomp} in an analogous manner.
\end{rem}

\subsection{Stability}\label{sec:stab}

\begin{thm}\label{thm:stab}
Let $R>0, \ s \ge 0$. There exists $\tau_R>0$ and $\sigma>d/2$ such that for any $ \tau \le \tau_R$ and $w,v \in X^\sigma$, such that $\| w\|_{\sigma} \le R$ and $\|v \|_{\sigma} \le R$ we have,
$$
\|\varphi^\tau(v) - \varphi^\tau(w)\|_{s} \le e^{\tau C_R}\|v-w\|_{s},
$$
where $C_R$ denotes a generic constant depending on $R$ (and on $s$). 
\end{thm}
\begin{proof}[Proof of Theorem \ref{thm:stab}]
Using the second estimate in equation \eqref{A2.2} we have
$$
\| \varphi^\tau(v) - \varphi^\tau(w)\|_{s} \le\left(1+ \tau C_{s}(\|v\|_{\sigma}, \|w\|_{\sigma})\right)\|v-w\|_{s} + \tau C_{s}(\| \varphi^\tau(v)\|_{\sigma}, \| \varphi^\tau(w)\|_{\sigma}) \| \varphi^\tau(v) -  \varphi^\tau(w) \|_{s},
$$
with $\sigma = d/2 + \epsilon$ if $s \le d/2$ and $\sigma = s$ if $s > d/2$.
By Theorem \ref{thm:fpt} we have that there exists $\tau_R > 0$ such that for all $\tau \le \tau_{R}$ we have the bounds: $\|\varphi^\tau(v)\|_{\sigma} \le 2R$ and $\|\varphi^\tau(w)\|_{\sigma} \le 2R $. Hence, it follows from the above that,
$$
\| \varphi^\tau(v) - \varphi^\tau(w)\|_{s} \le \frac{(1+\tau C_s(R,R))}{(1-\tau C_s(2R,2R))}\|v-w\|_{s} \le e^{\tau C_R}\|v-w\|_{s},
$$
for some $C_R > 0$.
\end{proof}

It remains to combine the stability argument presented in Section \ref{sec:stab} together with the  local error bounds of Section \ref{sec:local} to prove the global convergence result stated in Proposition \ref{prop:toProve}.

\begin{proof}[Proof of Proposition \ref{prop:toProve}] We let $e^n= u^n - u(t_n)$, with $e^0 = 0$.
First, thanks to Proposition \ref{prop:R1bd} and \ref{prop:R23} we have the local error bound in $L^2$ required for the global convergence analysis, where the regularity requirements on $u$ have been optimized depending on the fractional order of convergence desired. In order to apply the stability bound stated in Theorem \ref{thm:stab} and to conclude with a Lady Windermere's fan type argument, one needs to show the following uniform bound on the numerical solution,
\be\label{bound_sigma}
\|u^n\|_{\sigma} \le M_T, \quad \forall n\tau \le T,
\ee 
for some $\sigma > d/2$ and $M_T>0$.

We let $\sigma = d/2 + \epsilon$, for some small $\epsilon > 0$. To obtain the bound \eqref{bound_sigma} we show that there exists $\delta>0$, a constant $C_{R_n} = C_T(\|u^n\|_{\sigma})$ depending on $\| u^n\|_{\sigma}$ and on $\sup_{[0,T]}\|u(t) \|_{\sigma}$, and some $\tau_{R_n} >0$ also depending on $\|u^n\|_{\sigma}$ such that the following global error bound is met,
\be\label{induc-sigma}
\|e^{n+1}\|_{\sigma} \le \|\varphi^\tau(u(t_n))-u(t_{n+1})\|_{\sigma} + \|\varphi^\tau(u^n) -\varphi^\tau(u(t_n)) \|_{\sigma}
\le 
C_{T,\gamma} \tau^{1+\delta} + e^{\tau C_{R_n}}\|e^n\|_{\sigma}
\ee
for all $\tau \le \tau_{R_{n}}$.
One can obtain the second term in the above estimate for $\tau \le \tau_{R_n}$, where $\tau_{R_n}$ depends on $\| u_n\|$ and on $\sup_{[0,T]}\|u(t) \|_{\sigma}$, by applying Theorem \ref{thm:stab} with $s=\sigma$. Hence, it remains to obtain the first term in the above estimate, which corresponds to the local error bound in $X^\sigma$. Namely, by letting $\mathcal{R}(\tau,t_n) = u(t_{n+1}) - \varphi^\tau(u(t_n))$ we show that there exists $\delta >0$ such that 
\be\label{R_sigma_0}
\| \mathcal{R}(\tau,t_n)\|_{\sigma} \le C_{T,\gamma}\tau^{1+\delta},
\ee
where $C_{T,\gamma}$ is a constant depending on $T$ and on the regularity assumptions on $u$ (which in turn depend on $\gamma$, the fractional order of convergence required).
We establish the above local error estimate \eqref{R_sigma_0} by using the following interpolation bound,
\begin{align}\label{interpolation}
||\mathcal{R}(\tau,t_n)||_{\sigma} \le 
\left\{
\begin{array}{ll}
\| \mathcal{R}(\tau,t_n)\|^\theta \| \mathcal{R}(\tau,t_n)\|^{1-\theta}_{2\gamma + 1} \quad & \text{if} \ \gamma > \frac{d}{4}\vspace{1mm} \\ 
\| \mathcal{R}(\tau,t_n)\|^{\widetilde{\theta}} \| \mathcal{R}(\tau,t_n)\|^{1-\widetilde{\theta}}_{\gamma +1 + \frac{d}{4}}  \quad & \text{if} \ 0 \le \gamma < \frac{d}{4} 
\end{array} \right. ,
\end{align}
where $(\theta,\widetilde{\theta}) \in (0,1)^2$ satisfies $\sigma = (1-\theta)(2\gamma +1) = (1-\widetilde \theta)(\gamma + 1 + d/4)$. We have already established the $L^2$-bound on $\mathcal{R}(\tau,t_n)$ in Section \ref{sec:local}, which is given by, 
\be\label{L2locbd}
\| \mathcal{R}(\tau,t_n)\| \le \tau^{2+\gamma} 
\left\{
\begin{array}{ll}
C_{T,\gamma}\left(\sup_{t\in [0,T]}\|u(t) \|_{{2\gamma +1}}\right)\quad & \text{if} \ \gamma > \frac{d}{4}\\
C_{T,\gamma}\left(\sup_{t\in [0,T]}\|u(t) \|_{{\gamma +1 + \frac{d}{4}}}\right)  \quad& \text{if} \ 0 \le \gamma < \frac{d}{4} 
\end{array} \right. 
.
\ee
To obtain the $X^{2\gamma +1}$ and $X^{\gamma +1 + \frac{d}{4}}$ bound on $\mathcal{R}(\tau,t_n)$ we simply express the local error using Duhamel's formula and the scheme \eqref{num},
\begin{align*}
\mathcal{R}(\tau,t_n) &= \int_0^\tau e^{i(\tau-s)\Delta}f(u(t_n+s), \overline{u}(t_n+s)) ds - i \frac{\tau}{2} e^{i\tau \Delta}\left( (u(t_n))^2 \varphi_1(-i \tau\Delta)\bar{u}(t_n) \right)\\
&\qquad\qquad - i\frac{\tau}{2} \left( (\varphi^\tau(u(t_n)))^2\varphi_1(i\tau \Delta)\overline{\varphi^\tau(u(t_n)} \right).
\end{align*}
One can bound each of the above terms separately using the first estimate in equation \eqref{A2.2} with $r = 2\gamma +1$ and $r = \gamma +1 + \frac{d}{4}$ together with equation \eqref{implBd} (with $R=\sup_{[0,T]} \|u(t) \|_{\sigma}$) to obtain that there exists some $\tilde\tau_0 > 0$ depending on $u_0$ and $T$ such that for all $\tau\le \tilde\tau_0$
\begin{align}\label{local-r_1}
||\mathcal{R}(\tau,t_n)||_{r} &\le C(\sup_{[0,T]}||u(t)||_{\sigma}, \sup_{[0,T]}||u(t)||_{r}) \tau
\le C_{T,r} \tau.
\end{align}
We conclude that the bound \eqref{R_sigma_0} follows from equation \eqref{interpolation} (with $\delta = (1+\gamma)\theta$ for $\gamma > d/4$ and $\delta = (1+\gamma)\widetilde\theta$ for $\gamma < d/4$), and where the constant $C_{T,\gamma}$ is given in equation \eqref{L2locbd} by the $L^2$ local error bound.
We then proceed by induction on \eqref{induc-sigma} to obtain that there exists a $\tau_0>0$ which depends on $T $ and $u_0$ for which the uniform bound \eqref{bound_sigma} is true for all $\tau\le \tau_0$.

Finally, by taking $s=0$, $\sigma = d/2 + \epsilon$, and $R = \max\{M_T, \sup_{[0,T]}\| u(t)\|_{\sigma} \}$ in Theorem \ref{thm:stab} yields the existence of a $\tau_R $ which depends on $T$ and $u_0$ such that for all $\tau\le \tau_R$,
$$
\| e^{n+1}\| \le C_{T,\gamma}\tau^{2+\gamma} + e^{\tau C_R}\| e^n\|, \quad n\tau \le T,
$$
where $C_{T,\gamma}$ is given in equation \eqref{L2locbd}. The global error bound of Proposition \ref{prop:toProve} follows by iterating the above estimate and taking $\tau_{\text{min}} = \min \{ \tau_0, \tau_R\}$, which concludes the proof.
\end{proof}
\subsection*{Acknowledgements}
{\small
This project has received funding from the European Research Council (ERC) under the European Union's Horizon 2020 research and innovation programme (grant agreement No. 850941).
The author would like to express her thanks to Katharina Schratz for her guidance, as well as David Lee, Georg Maierhofer, and Fr\'ed\'eric Rousset  for helpful discussions.
}  


\begin{thebibliography}{99}

\bibitem{adams}
R.~Adams, J.J.F.~Fournier {\it Sobolev Spaces}, Springer, 2003.

\bibitem{YGP} Y.~Alama Bronsard, \textit{Error analysis of a class of semi-discrete schemes for solving the Gross-Pitaevskii equation at low regularity}, J. Comp. App. Math, 114632, July 2022, doi: 10.1016/j.cam.2022.114632.

\bibitem{ABBSBdd}
Y. ~Alama Bronsard, Y.~Bruned, K.~Schratz, {\it Low regularity integrators via decorated trees},  preprint (2022), arxiv.org/abs/2202.01171.

\bibitem{Antil}
H.~Antil, J.~Pfefferer, S.~Rogovs, {\it Fractional Operators with Inhomogeneous Boundary Conditions:
Analysis, Control, and Discretization}, Communication in Mathematical Sciences, 2017, 16(5): doi:10.4310/CMS.2018.v16.n5.a11.

\bibitem{lowNeumann}
G.~ Bai, B.~ Li, and Y.~ Wu, {\it A constructive low-regularity integrator for the 1d cubic nonlinear Schr\"odinger equation under the Neumann boundary condition.} IMA J. Numer. Anal. doi: 10.1093/imanum/drac075.

\bibitem{VMS}
V.~Banica, G.~Maierhofer, K.~Schratz, {\it Numerical integration of Schr\"odinger maps via the Hasimoto transform} preprint (2022), arxiv.org/abs/2205.05024.


\bibitem{BG}
J.~Bernier, B.~Gr\'ebert, {\it Birkhoff normal forms for Hamiltonian PDEs in their energy space}. Journal de l'\'Ecole polytechnique Math\'ematiques, Tome 9 (2022), pp. 681-745. doi:10.5802/jep.193.

\bibitem{Besse}
C. ~Besse, {\it A relaxation scheme for the nonlinear Schr\"odinger equation}, SIAM J. Numer. Anal. 42(3) (2004), 934--952.

\bibitem{Vaz}
M.~ Bonforte, Y.~Sire, J L.~V\'azquez,
{\it Existence, Uniqueness and Asymptotic behaviour for fractional porous medium equations
on bounded domains}, Discrete and Continuous Dynamical Systems, 2015, 35(12): 5725-5767. doi: 10.3934/dcds.2015.35.5725.

\bibitem{BS}
Y.~Bruned, K.~Schratz, \textit{Resonance based schemes for dispersive equations via decorated trees}, to appear in Forum of Mathematics, Pi (2022).

\bibitem{CLLsecondNLS}
J. Cao, B. Li, and Y. Lin. {\it A new second-order low-regularity integrator for the cubic nonlinear Schr\"odinger equation}, IMA Journal of Numerical Analysis (2023): drad017.

\bibitem{CS_KG}
M. Cabrera Calvo, K. Schratz. {\it Uniformly accurate low regularity integrators for the Klein-Gordon equation from the classical to non-relativistic limit regime}, SIAM Journal on Numerical Analysis, 60(2), 888-912.

\bibitem{SymmExp}
E. ~Celledoni, D. ~Cohen, B. ~Owren, {\it Symmetric exponential integrators with an application to the cubic Schr\"odinger equation}, Foundations of Computational Mathematics, 8 (2008), pp. 303--317.

\bibitem{CHL}
D.~Cohen, E.~Hairer, C.~Lubich, {\it Conservation of energy, momentum and actions in numerical discretizations of non-linear wave equations.} Numerische Mathematik, 110(2):113--143, 2008.

\bibitem{fracSplit}
J.~ Eilinghoff, R.~ Schnaubelt, K.~ Schratz, {\it Fractional error estimates of splitting schemes for the nonlinear Schr\"odinger equation.} J. Math. Anal. Appl. 442, 740-760 (2016).

\bibitem{Faou}
E.~Faou, {\it Geometric Numerical Integration and Schr\"odinger Equations.} European Math. Soc. Publishing House, Z\"urich 2012.

\bibitem{FGP}
E.~Faou, B.~Gr\'ebert, E.~Paturel, {\it Birkhoff normal form for splitting methods applied to semilinear Hamiltonian PDEs. Part I. Finite-dimensional discretization.} Numer. Math. 114, 429?458 (2010). doi: 10.1007/s00211-009-0258-y.

\bibitem{GL}
L.~ Gauckler, C.~Lubich, {\it Splitting Integrators for Nonlinear Schr\"odinger Equations Over Long Times} FoCM, 10: 275--302, 2010.

\bibitem{HLW} E. Hairer, C. Lubich, G. Wanner, {\it Geometric Numerical Integration. Structure-Preserving Algorithms for Ordinary Differential Equations.} Second Edition. Springer Berlin (2006).

\bibitem{Peterseim}
P. Henning and D. Peterseim, {\it Crank-Nicolson Galerkin approximations to nonlinear Schr\"odinger equations with rough potentials.} M3AS Math. Models Methods Appl. Sci. 27(11):2147-2184, (2017).

\bibitem{ExpInt}
M. Hochbruck, A. Ostermann, {\it Exponential integrators}, Acta Numerica, 19, 209-286. doi:10.1017/S0962492910000048 (2010).

\bibitem{HofS}
M. Hofmanov\'a, K. Schratz, {\it An oscillatory integrator for the KdV equation}, Numer. Math. 136:1117-1137 (2017).

\bibitem{hormander}
L.~H\"ormander, {\it Lectures on Nonlinear Hyperbolic Differential Equations}, Springer, 2003.

\bibitem{NS}
B. Li, S. Ma, K. Schratz. {\it A semi-implicit low-regularity integrator for Navier-Stokes equations}, to appear in SIAM J. Numer. Anal., arXiv:2107.13427 (2021).

\bibitem{LW21}
B. Li, Y. Wu, {\it A fully discrete low-regularity integrator for the 1D periodic cubic nonlinear Schr\"odinger equation}, Numerische Mathematik 149.1 (2021): 151-183.

\bibitem{LW22}
{ \rm B. Li, Y. Wu},
\newblock {\it An unfiltered low-regularity integrator for the KdV equation with solutions below $H^{1}$}.
https://arxiv.org/abs/2206.09320.


\bibitem{L}
C.~ Lubich, {\it On splitting methods for Schr\"odinger-Poisson and cubic nonlinear Schr\"odinger equations},
Math. Comput., 77, pp. 2141--2153 (2008).

\bibitem{GK}
G.~Maierhofer, K.~Schratz. {\it Bridging the gap: symplecticity and low regularity on the example of the KdV equation.} preprint (2022), arxiv.org/abs/2205.05024.

\bibitem{mKdV}
C.~Ning, Y.~ Wu, X.~Zhao. {\it An Embedded Exponential-Type Low-Regularity Integrator for mKdV Equation}. SIAM Journal on Numerical Analysis, 60(3), 999-1025  (2022).


\bibitem{OS1}
A.~Ostermann, K.~Schratz, 
\textit{Low regularity exponential-type integrators for semilinear Schr\"odinger equations}, Found Comput Math 18, 731--755 (2018). https://doi.org/10.1007/s10208-017-9352-1

\bibitem{ORS1}
A.~Ostermann, F.~Rousset, K.~Schratz,
\textit{Error estimates of a Fourier integrator for the cubic Schr\"odinger equation at low regularity}, Found Comput Math 21, 725--765 (2021). https://doi.org/10.1007/s10208-020-09468-7

\bibitem{ORS2}
A. ~Ostermann, F. ~Rousset, K. ~Schratz, {\it Fourier integrator for periodic NLS: low regularity estimates via discrete Bourgain spaces}, to appear in J. Eur. Math. Soc. (JEMS).
http://arxiv.org/abs/2006.12785 


\bibitem{OWY}
A. Ostermann, Y. Wu, F. Yao. {\it A second-order low-regularity integrator for the nonlinear Schr\"odinger equation,}
Adv Cont Discr Mod 2022, 23 (2022). https://doi.org/10.1186/s13662-022-03695-8.

\bibitem{RSKDV} 
{\rm F. Rousset, K. Schratz,} {\em Convergence error estimates at low regularity for time discretizations of KdV,} 
to appear in {\em Pure and Applied Analysis}. https://arxiv.org/abs/2102.11125.

\bibitem{RS}
F.~Rousset, K.~Schratz, \textit{A general framework of low regularity integrators}, SIAM J. Numer. Anal., 2021, 59:1735-1768. doi: 10.1137/20M1371506.

\bibitem{S_Dirac}
K. Schratz, Y. Wang, X. Zhao, {\it Low-regularity integrators for nonlinear Dirac equations}, Math. Comp. 90:189-214 (2021).

\bibitem{VT}
V. Thom\'ee, {\it Galerkin Finite Element Methods for Parabolic Problems}. Springer-Verlag, Heidelberg, 2006.

\bibitem{WuYao22}
Y. Wu, F. Yao. {\it A first-order Fourier integrator for the nonlinear Schr\"odinger equation on $\mathbb{T}$ without loss of regularity}, Mathematics of Computation 91.335 (2022): 1213-1235.

\bibitem{WZ22}
{\rm Y. Wang, X. Zhao},
{\em A symmetric low-regularity integrator for nonlinear Klein-Gordon equation,} {\em Math. Comp.} 91, (2022) 2215-2245. https://doi.org/10.1090/mcom/3751

\bibitem{WZhao22}
{\rm Y. Wu, X. Zhao},
{\em Embedded exponential-type low-regularity integrators for KdV equation under rough data,} {\em BIT Numer. Math.} 62, (2022) 1049--1090.
https://doi.org/10.1007/s10543-021-00895-8.

\bibitem{atsuchi}
A.~Yagi, {\it Abstract Parabolic Evolution Equations and their Applications}. Springer-Verlag Berlin Heidelberg, 2010.

\end{thebibliography}
\end{document}